\documentclass[11pt]{amsart}
\pdfoutput=1
\usepackage[varg]{txfonts}
\usepackage{amsmath, amssymb,amsaddr}
\usepackage{amsfonts}
\usepackage{mathrsfs}
\usepackage{bbm}
\usepackage{url}
\usepackage[arrow,matrix,curve,cmtip,ps]{xy}
\usepackage{graphicx}
\usepackage{amsthm}
\usepackage{float}
\usepackage{amsthm}
\usepackage[utf8]{inputenc}
\usepackage[T1]{fontenc}
\usepackage[utf8]{inputenc}
\usepackage{multirow}
\usepackage{mathtools}
\usepackage[pagebackref=true,pdfauthor={ShinnYih Huang},citecolor=blue,colorlinks=true,linkcolor=blue]{hyperref}

\allowdisplaybreaks

\newtheorem{theorem}{Theorem}[section]
\newtheorem{lemma}[theorem]{Lemma}

\newtheorem{corollary}[theorem]{Corollary}
\newtheorem{conjecture}[theorem]{Conjecture}

\newtheorem*{theorem*}{Theorem}
\theoremstyle{remark}
\newtheorem{remark}[theorem]{Remark}

\newtheorem{question}{Question}

\numberwithin{equation}{section}

\begin{document}
\title[ Minimally Intersecting Filling Pairs]{ Minimally Intersecting Filling Pairs on Surfaces}

\author{Tarik Aougab, Shinnyih Huang}

\address{Department of Mathematics \\ Yale University \\ 10 Hillhouse Avenue, New Haven, CT 06510 \\ USA}
\email{tarik.aougab@yale.edu, shinnyih.huang@yale.edu}

\date{\today}


\begin{abstract}

Let $S_{g}$ denote the closed orientable surface of genus $g$. We construct exponentially many mapping class group orbits of pairs of simple closed curves which fill $S_{g}$ and intersect minimally, by showing that such orbits are in correspondence with the solutions of a certain permutation equation in the symmetric group. Next, we demonstrate that minimally intersecting filling pairs are combinatorially optimal, in the sense that there are many simple closed curves intersecting the pair exactly once. We conclude by initiating the study of a topological Morse function $\mathcal{F}_{g}$ over the Moduli space of Riemann surfaces of genus $g$, which, given a hyperbolic metric $\sigma$, outputs the length of the shortest, minimally intersecting filling pair for the metric $\sigma$. We completely characterize the global minima of $\mathcal{F}_{g}$, and using the exponentially many mapping class group orbits of minimally intersecting filling pairs that we construct in the first portion of the paper, we show that the number of such minima grow at least exponentially in $g$. 

\end{abstract}

\maketitle

\section{Introduction}

Let $S_{g,b}$ denote the orientable surface of genus $g$ with $b$ boundary components, and let $\mbox{Mod}(S_{g,0})$ denote the \textit{mapping class group} of the closed surface $S_{g,0}$, the group of all orientation preserving homeomorphisms of $S_{g,0}$, modulo isotopy (see Farb-Margalit \cite{Far-Mar} for background). 

A pair of curves on $S_{g}$ (shorthand for $S_{g,0}$) are said to \textit{fill} the surface if the complement of their union is a disjoint union of topological disks. 

Let $\mathcal{I}(S_{g})$ denote the set of all isotopy classes of simple closed curves on $S_{g}$. Then $\mbox{Mod}(S_{g})$ acts coordinate-wise on the product $\mathcal{I}(S_{g})\times \mathcal{I}(S_{g})$, and this action descends to an action on the quotient
\[\mathcal{P}(S_{g}):=\mathcal{I}(S_{g})\times \mathcal{I}(S_{g})/[(\alpha,\beta)\sim (\beta,\alpha)].\]

Our first result establishes the exponential growth of the $\mbox{Mod}(S_{g})$-orbits of $\mathcal{P}(S_{g})$ which fill and intersect minimally:

\begin{theorem} \label{thm:1}
There exists a function $f(g) \sim 3^{g/2}/g^{2}$ satisfying the following: Let $N(g)$ denote the number of $\mbox{Mod}(S_{g})$-orbits of minimally intersecting filling pairs on $S_{g}$. Then  

\[  f(g) \leq N(g)\leq 2^{2g-2}(4g-5)(2g-3)!.\]
\end{theorem}

In the statement of Theorem \ref{thm:1}, $\sim$ denotes the following growth-rate relation:
\[ r(g) \sim s(g) \Leftrightarrow 0< \lim_{g\rightarrow \infty} \frac{r(g)}{s(g)}< \infty.\]

We compare Theorem \ref{thm:1} to Theorem $1$ of Anderson-Parlier-Pettet \cite{And-Par-Pet}, which says that if $\Omega =\left\{\alpha_{1},...,\alpha_{n}\right\}$ is a filling set of simple closed curves on $S_{g}$ pairwise intersecting no more than $K$ times, then 
\[ n^{2}-n \geq \frac{4g-2}{K}.\]

Thus, when $n= 2, K\leq 2g-1$, and indeed we show in section $2$ that a minimally intersecting filling pair intersects $2g-1$ times if and only if $g\neq 2$ (when $g=2$, $4$ intersections is minimal).

Given a pair of filling curves $(\alpha,\beta)$, there are potentially many ways of defining a sort of combinatorial ``complexity'' for the pair. One such candidate is to count the number of simple closed curves intersecting the filling pair a small number of times. Heuristically, one expects a minimally intersecting filling pair to only ``barely'' fill $S_{g}$. Therefore, there should be many simple closed curves intersecting the union $\alpha\cup \beta$ only once. Indeed, we show:

\begin{theorem}\label{thm:2}
Let $(\alpha,\beta)$ be a filling pair on $S_{g}, g>2$ and define $T_{1}(\alpha,\beta) \in \mathbb{N}$ to be the number of simple closed curves intersecting $\alpha \cup \beta$ only once. Then $T_{1}(\alpha,\beta)\leq 4g-2$, with equality if $(\alpha,\beta)$ is minimally intersecting. 
\end{theorem}

Let $\mathcal{M}(S_{g})$ denote the \textit{moduli space} of Riemann surfaces of genus $g$, which is identified with the space of all complete, finite volume hyperbolic metrics on $S_{g}$ up to isometry. The \textit{systole} function $\mathcal{R}_{g}:\mathcal{M}(S_{g})\rightarrow \mathbb{R}_{+}$ outputs the length of the shortest closed geodesic for a given hyperbolic metric $\sigma\in \mathcal{M}(S_{g})$. 

$\mathcal{R}_{g}$ has been studied extensively, (see Buser-Sarnak \cite{Bus-Sar}; Parlier \cite{Par1}, \cite{Par2}; and Schaller \cite{Schal1},\cite{Schal2}, \cite{Schal3}) in part due to the fact that it is a \textit{topological Morse function}  \cite{Akro}, a generalization of the notion of a classical Morse function to functions which are not necessarily smooth, and for which versions of the Morse inequalities also hold. Thus a careful analysis of the critical points of $\mathcal{R}_{g}$ can yield naturally arising cellular decompositions of $\mathcal{M}(S_{g})$. 

Motivated by this program, we initiate the study of another topological Morse function $\mathcal{F}_{g}$, the \textit{filling pair} function, which outputs the length of the shortest minimally intersecting filling pair with respect to a metric $\sigma \in \mathcal{M}(S_{g})$. Here, the length of a pair of curves is simply the sum of the individual lengths. We show:

\begin{theorem} \label{thm:3}
$\mathcal{F}_{g}$ is proper and a topological Morse function. For any $\sigma \in \mathcal{M}_{g}$, 
\[ \mathcal{F}_{g}(\sigma)\geq \frac{m_{g}}{2},\]
where 
\[m_{g}= (8g-4)\cdot \cosh^{-1}\left( 2 \left[\cos\left(\frac{2\pi}{8g-4}\right)+\frac{1}{2}\right] \right)  \]
denotes the perimeter of a regular, right-angled $(8g-4)$-gon. Furthermore, define 
\[ \mathcal{B}_{g}:= \left\{\sigma \in \mathcal{M}_{g} : \mathcal{F}_{g}(\sigma) = \frac{m_{g}}{2}\right\};\]
then $\mathcal{B}_{g}$ is finite and grows at least exponentially in $g$. If $\sigma\in \mathcal{B}_{g}$, the injectivity radius of $\sigma$  is at least  
\[\frac{1}{2} \cosh^{-1}\left(\frac{9}{\sqrt{73}}\right) \approx .3253 .\]
\end{theorem}

\textbf{Organization of paper and Acknowledgements.} In section $2$, we prove Theorem \ref{thm:1}; in section $3$, we prove Theorem \ref{thm:2}; and in section $4$, we prove Theorem \ref{thm:3}.

The authors would like to thank Yair Minsky for carefully reading through a draft of this project and for his many helpful suggestions. The first author would also like to thank Dan Margalit, Chris Arettines, Subhojoy Gupta, Babak Modami, and Jean Sun for many enlightening conversations regarding this work.

\section{Counting Minimally Intersecting Filling Pairs}

In this section we prove Theorem \ref{thm:1}. We begin with some preliminary notions and notation. All curves will be simple and closed. Let $\sim$ denote the free homotopy relation. 

Let $\chi(S_{g})= 2-2g$ denote the standard Euler characteristic.

If $\alpha,\beta$ are two curves on $S_{g}$, we define the \textit{geometric intersection number} of the pair $(\alpha,\beta)$, denoted $i(\alpha,\beta)$, by 
\[ i(\alpha,\beta)= \min_{\alpha' \sim \alpha}|\alpha' \cap \beta|, \]
where the minimum is taken over all curves $\alpha'$ freely homotopic to $\alpha$. 

If $|\alpha \cap \beta|= i(\alpha, \beta)$, we say that $\alpha$ and $\beta$ are in \textit{minimal position}.

\begin{lemma}\label{lm:1}
Suppose $(\alpha,\beta)$ is a pair of curves which fill $S_{g}, g \geq 1$. Then 
\[ i(\alpha,\beta)\geq 2g-1. \]
\end{lemma}

\begin{proof} Assume $\alpha,\beta$ are in minimal position. That $(\alpha,\beta)$ fills $S_{g}$ implies that the complement $S_{g}\setminus (\alpha\cup \beta)$ is a disjoint union $D$ of topological disks. Thus we can view the $4$-valent graph $\alpha \cup \beta$ as the $1$-skeleton of a cellular decomposition of $S_{g}$. 

The number of $0$-cells is $i(\alpha,\beta)$, and $4$-valency of $\alpha\cup \beta$ implies that the number of edges is $2i(\alpha,\beta)$. Therefore

\[ \chi(S_{g})= 2-2g= i(\alpha,\beta) - 2i(\alpha,\beta)+ |D| \]
\[ \Rightarrow i(\alpha,\beta)= 2g-2+|D|. \]

Then since $|D|\geq 1$, we obtain the desired result. 

\end{proof}

We remark that the proof of Lemma \ref{lm:1} implies that $i(\alpha,\beta)=2g-1$ exactly when $S_{g}\setminus (\alpha\cup\beta)$ consists of a single disk.

\subsection{Filling Permutations}

Given a minimally intersecting filling pair $(\alpha, \beta)$, choose an orientation for $\alpha$ and $\beta$. Label the sub-arcs of $\alpha$ from $\left\{\alpha_{1},...,\alpha_{2g-1}\right\}$, separated by the $2g-1$ intersection points in accordance with the chosen orientation for $\alpha$ (and after choosing an initial arc for $\alpha$), and similarly for $\beta_{1},...,\beta_{2g-1}$. Once labeled, we call such a pair an \textit{oriented filling pair}. A filling pair without the extra structure of an orientation is an \textit{ordinary filling pair}. 

The action of $\mbox{Mod}(S_{g})$ on ordinary filling pairs lifts naturally to an action on the set of oriented filling pairs. Given an oriented filling pair, we obtain an ordinary one by forgetting the orientation.

Then associated to the oriented pair is a permutation $\sigma^{(\alpha,\beta)}\in \Sigma_{8g-4}$, the permutation group on $8g-4$ symbols, which we define as follows.

We will identify $\Sigma_{8g-4}$ as the permutations of the ordered set 
\[ A(g)= \left\{\alpha_{1},\beta_{1},\alpha_{2},\beta_{2},...,\alpha_{2g-1},\beta_{2g-1}, \alpha_{1}^{-1},\beta_{1}^{-1},..., \alpha_{2g-1}^{-1},\beta_{2g-1}^{-1} \right\} .\]

Cutting along $\alpha\cup \beta$ produces an $(8g-4)$-gon $P(\alpha,\beta)$, whose edges correspond to elements of $A(g)$. Pick an initial edge of $P(\alpha, \beta)$ and orient it clockwise. Then $\sigma^{(\alpha,\beta)}$ is defined by 

\[ \sigma^{(\alpha,\beta)}(j):=k \Leftrightarrow \mbox{the $j^{th}$ element of $A(g)$ coincides with the $(k-1)^{st}$ edge of $P(g)$.  }\]

Note that $\sigma^{(\alpha,\beta)}$ must be an $(8g-4)$-cycle, and corresponds to the permutation of $A(g)$ obtained by rotating $P(\alpha,\beta)$ counter-clockwise by $2\pi/(8g-4)$. 

Let $Q_{g} \in \Sigma_{8g-4}$ denote the standard ordered $(8g-4)$-cycle, which in cycle notation is
\[ Q_{g} = (1,2,3,...,8g-4) .\]
Note also that $Q_{g}^{4g-2}$ sends an edge to itself, oriented in the opposite direction. Let $\tau_{g}\in \Sigma_{8g-4}$ denote the permutation 
\[ \tau_{g}= (1,3,5,...,4g-3)(2,4,6,...,4g-2)(8g-5,8g-7,...,4g-1)(8g-4,8g-6,...,4g), \]
which corresponds moving forward by one sub-arc in each of $\alpha,\beta$. 

In what follows, in order to simplify the relevant permutation equations, we will suppress the reference to $g$ and $(\alpha,\beta)$ when referring to a permutation (e.g., $Q_{g}$ will be written as $Q$ and $\sigma^{(\alpha,\beta)}$ as $\sigma$ ).  

We say that a permutation $\gamma$ \textit{respects parity}, or is \textit{parity respecting}, if 
\[i \equiv j (\mbox{mod}(2) ) \Leftrightarrow \gamma(i) \equiv \gamma(j) (\mbox{mod}(2)). \]

\begin{lemma} \label{lm:2}
Let $(\alpha,\beta)$ be a minimally intersecting filling pair on $S_{g}$. Then the permutation $\sigma= \sigma^{(\alpha,\beta)}$ satisfies the equation

\[   \sigma Q^{4g-2}\sigma= \tau \]

Conversely, if $\sigma\in \Sigma_{8g-4}$ is  a parity respecting $(8g-4)$-cycle which solves the above equation, then $\sigma$ defines a minimally intersecting filling pair on $S_{g}$. 
\end{lemma}

\begin{proof} Let $k\in \left\{1,...,8g-4\right\}$. Then $\sigma(k)$ is the labeling of the edge of $P(g)$ immediately following the edge labeled with the $k^{th}$ element of $A(g)$. The ordering of $A(g)$ has been chosen so that $Q^{4g-2}(j)$ is precisely the element of $A(g)$ identified with the inverse arc to $j$. Thus $\sigma   Q^{4g-2} \sigma $ sends $k$ to the next arc along whichever curve $k$ is apart of, with respect to the chosen orientation of that curve, which is equal to $\tau(k)$.

Conversely, let $\sigma$ satisfy the above permutation equation, and suppose further that $\sigma$ is an $(8g-4)$-cycle and is parity respecting. Then associated to $\sigma$ is a clockwise labeling of the edges of an $(8g-4)$-gon (up to cyclic permutation) from the set $\left\{1,...,8g-4\right\}$, such that edges alternate between receiving even and odd labels. This is simply the labeling obtained by choosing an initial edge of the polygon, and labeling edges in order from the initial edge in accordance with the permutation $\sigma$ written in cycle notation. 

This produces a labeling of the edges by the elements of $A(g)$, which we in turn interpret as a gluing pattern; each edge glues to its inverse. Thus we obtain a closed surface $S$, which is orientable because each edge is glued to one which is oppositely oriented. It remains to check that the Euler characteristic of the surface is $2-2g$, and that the boundary of the polygon projects to a minimally intersecting filling pair. 

 $\chi(S)=2-2g$ is equivalent to there being $2g-1$ equivalence classes of vertices under the gluing. Indeed, the permutation equation ensures that each equivalence class contains precisely $4$ vertices; as seen in Figure $1$, the vertex $v$ separating the edges labeled $\alpha_{k}$ and $\beta_{j+1}$ glues to the vertex separating the edges labeled $\beta_{j+1}^{-1}$ and $\alpha_{k+1}$. 

 \begin{figure}[H]
\centering
	\includegraphics[width=3.5in]{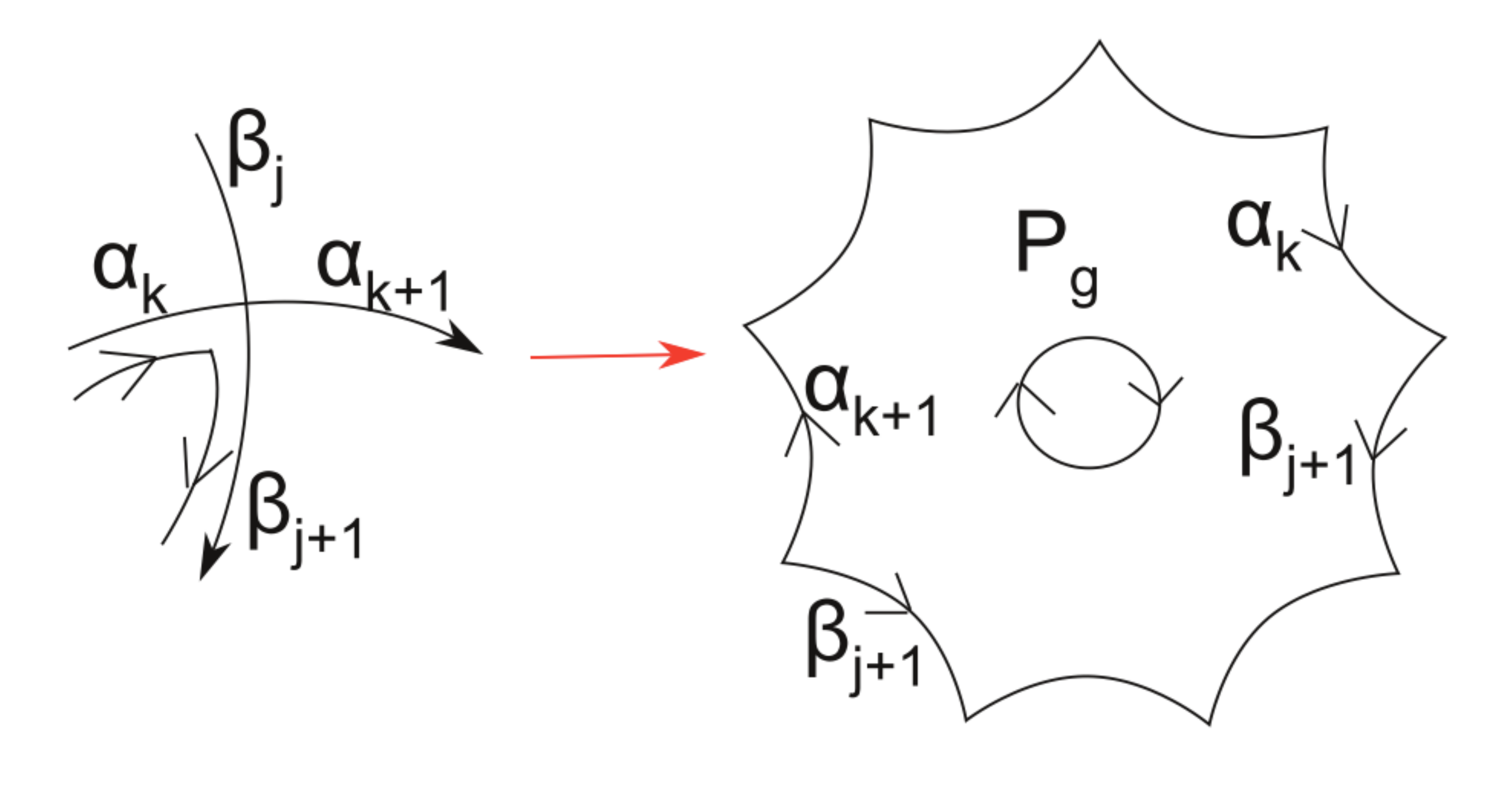}
\caption{ The permutation equation implies that if $\alpha_{k}$ and $\beta_{j+1}$ share a vertex in common along the boundary of the polygon $P(g)$, then $\beta_{j+1}^{-1}$ and $\alpha_{k+1}$ must also share a vertex in common, and these two vertices are glued together to obtain $S$. Iterating this argument shows that each gluing class contains $4$ vertices.}
\end{figure}

Similarly, the vertex separating $\beta_{j+1}^{-1}$ and $\alpha_{k+1}$ is glued to the vertex separating $\alpha_{k+1}^{-1}$ and $\beta_{j}^{-1}$, which in turn glues to the vertex separating $\beta_{j}$ and $\alpha_{k}^{-1}$. Finally, this fourth vertex glues to $v$ and the cycle is complete.

By the same analysis, the terminal vertex of $\alpha_{k}$ is identified to the initial vertex of $\alpha_{k+1}$, and therefore the $\alpha$-arcs are concatenated in the quotient to produce a single simple closed curve $\alpha$, and similarly for the $\beta$-arcs.  

\end{proof}

 We also note that if a permutation in $\Sigma_{8g-4}$ is an $(8g-4)$-cycle and parity respecting, it must send each even number to an odd number.

Henceforth, any permutation satisfying the requirements of Lemma \ref{lm:2} (being an $(8g-4)$-cycle, parity respecting, and satisfying the permutation equation) will be called a \textit{filling permutation}. 

Filling permutations parameterize oriented filling pairs. Therefore, in order to count ordinary filling pairs, it suffices to count filling permutations, and divide by the number of distinct orientations which a fixed ordinary filling pair admits. Lemma \ref{lm:3} characterizes which filling permutations correspond to the same ordinary filling pair:

\begin{lemma} \label{lm:3}
If $\Gamma_{1} =(\alpha_{1},\beta_{1}), \Gamma_{2}=(\alpha_{2},\beta_{2})$ are two minimally intersecting filling pairs on $S_{g}$ in the same $\mbox{Mod}(S_{g})$-orbit, then $\sigma^{(\alpha_{1},\beta_{1})}=\sigma^{(\alpha_{2},\beta_{2})}$, modulo conjugation by permutations of the form 
\[\mu_{g}^{l}\kappa_{g}^{k}\delta_{g}^{j}\eta_{g}^{i}:  l,i \in \left\{0,1\right\}; j,k\in \left\{0,1,...,2g-2\right\}, \]
where \[\kappa_{g}=(1,3,5,...,4g-3)(4g-1,4g+1,...,8g-7,8g-5),\]   
\[ \delta_{g}= (2,4,6,...,4g-2)(4g,4g+2,...,8g-4),\]
\[\eta_{g}= (1,4g-1)(3,4g+1)...(4g-3,8g-5), \mu_{g}= (1,2)(3,4)...(8g-5,8g-4). \]
\end{lemma}

\begin{proof} If $\Gamma_{1}$ and $\Gamma_{2}$ are \textit{oriented} filling pairs in the same $\mbox{Mod}(S_{g})$ orbit, there is a homeomorphism of $S_{g}$ taking $\Gamma_{1}$ to $\Gamma_{2}$, which lifts to a homeomorphism of the $(8g-4)$-gon that is simplicial in the sense that it fixes the boundary and sends edges to edges and vertices to vertices. 

The only option for such a map is a dihedral symmetry of the regular Euclidean $(8g-4)$-gon; since the map must also be orientation preserving, this leaves rotation as the only option. Thus the filling permutation for $\Gamma_{2}$, written in cycle notation, is obtained from that of $\Gamma_{1}$ by a cyclic rotation, and thus determines the same underlying permutation of $\Sigma_{8g-4}$. 

However, in general there will exist oriented filling pairs which are not in the same $\mbox{Mod}(S_{g})$ orbit, but whose underlying ordinary filling pairs are. Nevertheless, the fiber of the map from oriented pairs to ordinary pairs is in $1-1$ correspondence with the set of distinct orientations with which one can equip an ordinary pair. Fixing an ordinary filling pair, any choice of orientation can be obtained from any other by simply choosing a different initial arc along $\alpha$ (corresponding to conjugating by a power of $\kappa_{g}$) or $\beta$ (conjugating by $\delta_{g}$), reversing the direction of $\alpha$ (conjugating by $\eta_{g}$), and interchanging $\alpha$ and $\beta$ (conjugating by $\mu_{g}$).

\end{proof}

We call permutations of the form $\mu_{g}^{l}\kappa_{g}^{k}\delta_{g}^{j}\eta_{g}^{i}$ \textit{twisting} permutations.

In light of Lemma \ref{lm:3}, to obtain a lower bound on $N(g)$, it suffices to find a lower bound on the set of  filling permutations, and divide by an upper bound on the number of twisting permutations. Since $\kappa_{g}$ and $\delta_{g}$ both have order $2g-1$ and $\eta_{g},\mu_{g}$ are involutions, there are at most $4\cdot (2g-1)^{2}$ such permutations.

\subsection{Proof of the Lower Bound} 

Let $(\alpha,\beta)$ be an oriented minimally intersecting filling pair on $S_{g}$, and label the intersections $v_{1},...,v_{2g-1}$ in accordance with the chosen direction along $\alpha$. We first describe a way of ``extending'' $(\alpha,\beta)$ to a minimally intersecting pair on $S_{g+2}$ (see Remark \ref{rk:1} for an explanation of why we do not extend to $S_{g+1}$). 

Fix some $k, 1\leq k \leq 2g-1$. To create a minimally intersecting filling pair on $S_{g+2}$, we first excise a small disk around $v_{k}$, yielding $S_{g,1}$. Doing so turns $\alpha$ and $\beta$ into arcs $\tilde{\alpha}, \tilde{\beta}$ which intersect $2g-2$ times and which fill $S_{g,1}$. To complete the construction, we will glue a decorated copy $Z$ of $S_{2,1}$ to $S_{g,1}$ along its boundary component; by ``decorated'', we mean that $Z$ will come equipped with a special pair of oriented arcs $(a,b)$, such that if we concatenate $\tilde{\alpha}$ with $a$, and $\tilde{\beta}$ with $b$, we obtain a minimally intersecting filling pair on $S_{g+2}$. 

We now describe the arcs $(a,b)$. Consider the $3$-component multicurve $m$ on $S_{2,0}$ pictured in red in Figure $2$ below. The curve $\mbox{w}$ intersects each component of $m$; let $T_{m}(\mbox{w})$ denote the right Dehn twist of $\mbox{w}$ around $m$. 

\begin{lemma} \label{lm:4}
$(T_{m}(\mbox{w}),w)$ fills $S_{2,0}$, and the complement of the pair consists of $4$ disks. Furthermore, when $(T_{m}(\mbox{w}),\mbox{w} )$ are drawn in minimal position, there exists an intersection point which is on the boundary of all of these complementary disks. 
\end{lemma}

\begin{figure}[H]
\centering
	\includegraphics[width=3.5in]{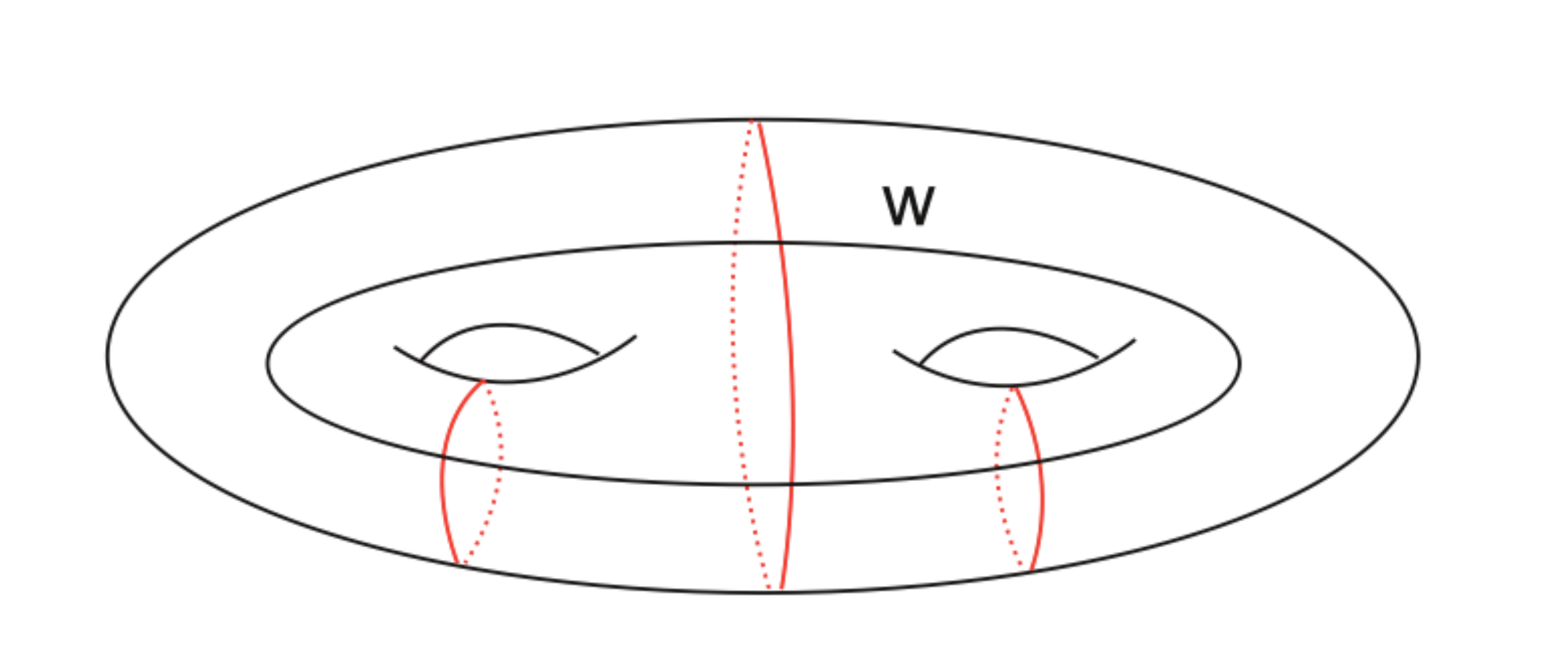}
\caption{The simple closed curve $\mbox{w}$ intersects each component of the red pants decomposition.}
\end{figure}

\begin{proof} The proof is completely constructive. $T_{m}(\mbox{w})$ and $\mbox{w}$ are pictured in Figure $3$ below ($T_{m}(\mbox{w})$ in red, $\mbox{w}$ in black). The representatives drawn in Figure $3$ intersect $6$ times; thus both curves are subdivided into $6$ arcs: $\left\{x_{1},...,x_{6}\right\}$ for $\mbox{w}$ and $\left\{y_{1},...,y_{6}\right\}$ for $T_{m}(\mbox{w})$, ordered cyclically with respect to some chosen orientation of $T_{m}(\mbox{w})$ (resp. $\mbox{w}$).

To examine the complement $S_{2,0}\setminus (T_{m}(\mbox{w})\cup \mbox{w})$, simply cut along the union of these two curves to obtain the $4$ labeled polygons in Figure $4$. $S_{2,0}$ is obtained by gluing these $4$ polygons back together along the oriented edge labelings. The four green vertices in Figure $4$ are all identified together in this gluing. 

\end{proof}

To complete the construction, we puncture $S_{2,0}$ at this green point by removing a small disk centered there to obtain $Z$, a copy of $S_{2,1}$; define the arc $a$ to be what remains of $\mbox{w}$, and what remains of $T_{m}(\mbox{w})$ is defined to be the second arc $b$. We then glue $Z$ to $S_{g,1}$ along the boundary of the disk we removed centered at the intersection point $v_{k}$, matching the endpoints of $a$ with those of $\tilde{\alpha}$ and the endpoints of $b$ with $\partial \tilde{\beta}$, obtaining two simple closed curves $(\alpha_{k}, \beta_{k})$ on $S_{g+2,0}$. Here, the subscript $k$ denotes the fact that we obtained this pair of curves by excising a disk around the intersection point $v_{k}$.

\begin{figure}[H]
\centering
	\includegraphics[width=3.5in]{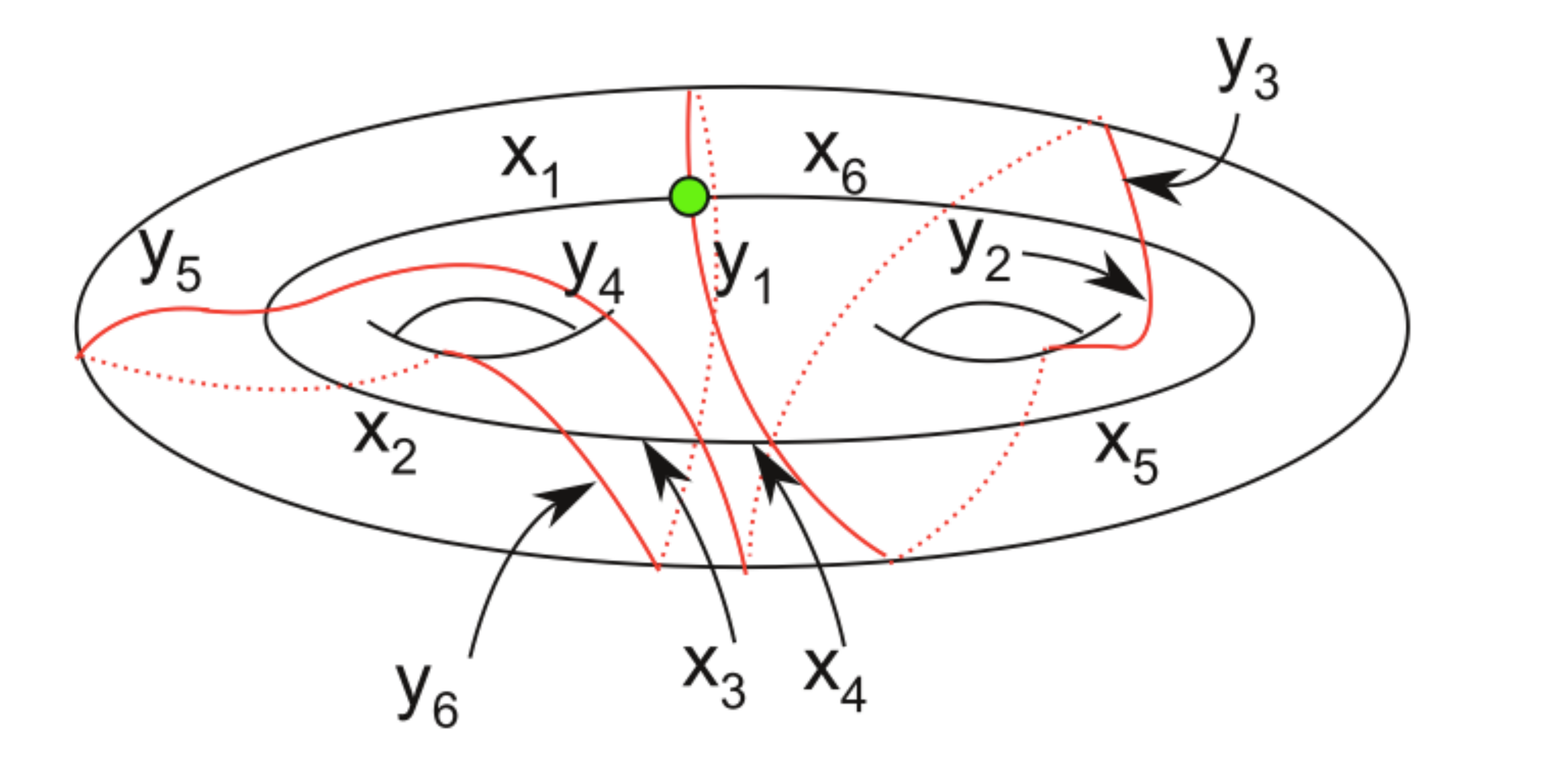}
\caption{$\mbox{w}$ in black and $T_{m}(\mbox{w})$ in red. The green point is the desired intersection point which is on the boundary of all $4$ complementary regions.}
\end{figure}

We claim that $(\alpha_{k}, \beta_{k})$ fill $S_{g+2,0}$. Assuming this, they must intersect minimally because by construction, $i(\alpha_{k},\beta_{k})\leq 2g+3= 2(g+2)-1$. Indeed, because the green point of intersection was on the boundary of all four complementary regions of $\mbox{w} \cup T_{m}(\mbox{w})$, the number of complementary regions of $S_{g+2} \setminus (\alpha_{k} \cup \beta_{k})$ is equal to the number of complementary regions of $S_{g} \setminus (\alpha \cup \beta)$. Therefore $S_{g+2}\setminus (\alpha_{k} \cup \beta_{k})$ is connected, and by the same basic Euler characteristic argument used in the proof of Lemma \ref{lm:1} the complement of $\alpha_{k}\cup \beta_{k}$ must be simply-connected. 

\begin{figure}[H]
\centering
	\includegraphics[width=3in]{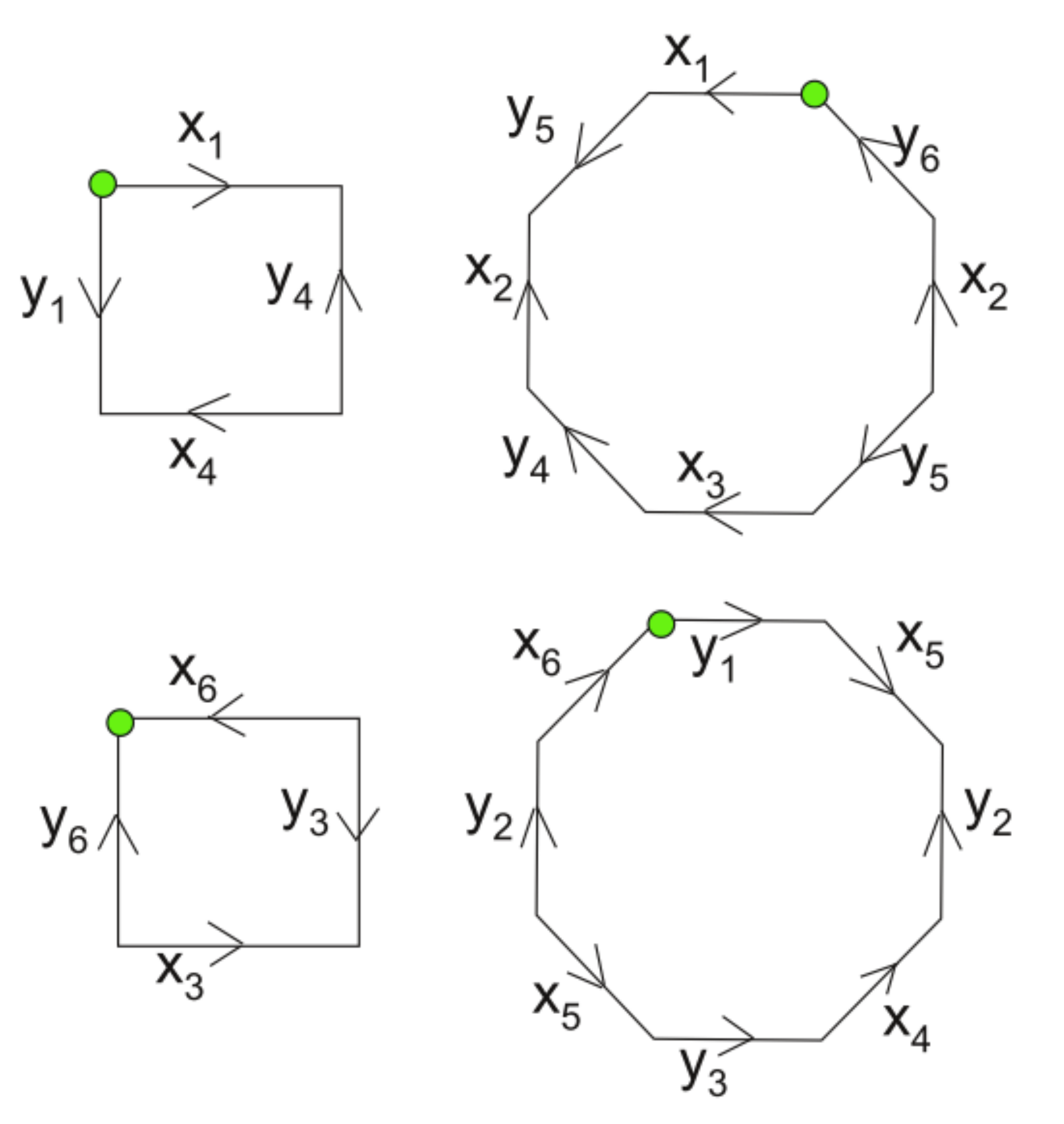}
\caption{Cutting along $T_{m}(\mbox{w})\cup \mbox{w}$ produces four complementary simply connected regions. When glued back up according to the oriented edge labelings, the four green vertices are identified.}
\end{figure}

\begin{remark} \label{rk:1}
It is unclear how to attempt such a construction by extending a minimally intersecting filling pair on $S_{g}$ to one on $S_{g+1}$. Our construction works by finding a pair of simple closed curves on $S_{2}$ intersecting $6$ times, and such that any other simple closed curve intersects their union at least twice. The key point here is that the arcs on $S_{2,1}$ fill in the \textit{arc and curve sense}, meaning that no essential arc or simple closed curve is disjoint from their union.

To extend a minimally intersecting filling pair on $S_{g}$ to one on $S_{g+1}$ in this fashion, we would need a pair of simple closed curves on $S_{1}$ intersecting $4$ times such that any other simple closed curve intersects their union at least twice. One way of guaranteeing this is to choose a pair of curves that are distance $3$ in the Farey graph. However, if a curve is distance $3$ or more from the $(1,0)$ curve, it must intersect $(1,0)$ at least $5$ times; since $SL(2,\mathbb{Z})$ acts transitively on the vertex set of the Farey graph, the same is true for any vertex.

\end{remark}

It remains to show that by choosing different intersection points $v_{j}$ to excise and glue a copy of $Z$ along the resulting boundary component, and by choosing different initial filling pairs on $S_{g}$ and carrying out this construction, we obtain distinct filling permutations. 

After gluing a copy of $Z$ to $S_{g}$ along a minimally intersecting filling pair $(\alpha,\beta)$ in the above fashion, we obtain a filling pair $(\alpha',\beta')$ on $S_{g+2}$ which contains a copy of the arcs $a,b$ as sub-arcs of $\alpha'$ and $\beta'$, respectively. 

 That is, there exists an open sub-arc $a'$ of $\alpha'$ whose closure has boundary consisting of two intersection points of $\alpha' \cup \beta'$, and such that:

\begin{enumerate}

\item it is comprised of $6$ consecutive arcs $\alpha'_{k},...,\alpha'_{k+5}$, (since $a'$ is open, we define $\alpha'_{k}, \alpha'_{k+5}$ to be open, and the other $4$ arcs to be closed); 
\item there is a sub-arc $b'$ of $\beta'$, comprised of $6$ consecutive $\beta'$-arcs, such that $|a'\cap b'|= 5$, and the combinatorics of these intersections completely coincide with those displayed between the arcs $a$ and $b$ in Figure $3$. 

\end{enumerate}

That is to say, the $j^{th}$ intersection point counted along $a'$ coincides with the $k^{th}$ intersection point along $b'$ if and only if this is the case for the arcs $a$ and $b$ on $Z$. We call $b'$ the \textit{companion arc} to $a'$. The union $a'\cup b'$ is called a \textit{$Z$-piece}. When convenient, we will alternate between thinking of a $Z$-piece as being a surface with boundary, or as being contained within a particular filling pair by identifying it with the arcs $a' \cup b'$. 

\begin{remark} \label{rk:2}
Given a $Z$-piece within a filling pair $(\alpha,\beta)$ with $\alpha$ and $\beta$ arcs labeled $x_{1},x_{2},...,x_{6}$ and $y_{1},...,y_{6}$, respectively as in Figure $3$, note that none of the following $4$ points can coincide on the surface:

\begin{enumerate}
\item the initial point of $x_{1}$;
\item the initial point of $y_{1}$;
\item the terminal point of $y_{6}$;
\item the terminal point of $x_{6}$;
\end{enumerate}

This observation is trivial, but it will play an important role in the proof of Lemma \ref{lm:6}.

\end{remark}

Note that there exists a filling pair $(\alpha(1),\beta(1))$ on $S_{1}$ intersecting $2g-1=1$ times, and at the end of this section we will demonstrate a filling pair $(\alpha(4),\beta(4))$ on $S_{4}$ intersecting $7$ times. If $(\alpha,\beta)$ on $S_{g}$ is a minimally intersecting filling pair obtained from one of these two``seed'' pairs by successively gluing on copies of $Z$, whether $(\alpha,\beta)$ is obtained from $(\alpha(1),\beta(1))$ or $(\alpha(4),\beta(4))$ in this fashion is completely determined by the parity of $g$. In the case that $g>4$ is even, we call $(\alpha,\beta)$ a $\frac{g-4}{2}$-\textit{child} of the \textit{ancestor} $(\alpha(4),\beta(4))$, and similarly if $g \geq 3$ is odd. 

The purpose of the next lemma is to show that $Z$-pieces are symmetric with respect to interchanging the roles of the arcs $a$ and $b$ in Figure $3$:

\begin{lemma} \label{lm:5}
There exists a homeomorphism of a $Z$-piece interchanging the arcs $a$ and $b$ and preserving the labeling, in the sense that sub-arc $x_{k}, 1\leq k \leq 6$ along $a$ maps to sub-arc $y_{k}$ along $b$. 
\end{lemma}
\begin{proof} Consider the polygonal decomposition in Figure $4$; if, for each $k$, one relabels $x_{k}$ as $y'_{k}$, and $y_{k}$ as $x'_{k}$, then the concatenation $x'_{1}\ast x'_{2} \ast...\ast x'_{6}$ displays the exact intersection combinatorics with the arc $y'_{1} \ast y_{2} \ast...\ast y'_{6}$, as the oriented arc $a$ does with $b$.  

\end{proof}

Given $Z_{1}$ a $Z$-piece, define $Z_{1}^{(\alpha)}$ to be the interior acs of $\alpha$ contained in $Z_{1}$; that is, $Z_{1}^{(\alpha)}$ is comprised of the sub-arcs $\left\{x_{2},x_{3},x_{4},x_{5}\right\}$. Define $Z_{1}^{(\beta)}$ similarly.

If $Z_{1},Z_{2}$ are two $Z$-pieces in $(\alpha,\beta)$, let $(Z_{1},Z_{2})^{\cap}$ denote the set of arcs
\[ (Z_{1},Z_{2})^{\cap}= (Z_{1}^{(\alpha)}\cap Z_{2}^{(\alpha)}) \cup (Z_{1}^{(\beta)}\cap Z_{2}^{(\beta)}).\]

For $1\leq k \leq 6$, let $x_{k}(Z_{1})$ denote the terminal point of the sub-arc $x_{k}$ in $a$ in the $Z$-piece $Z_{1}$, and define $y_{k}(Z_{1})$ similarly. 

The purpose of the next lemma is to show that distinct $Z$-pieces can not share any $\alpha$-arcs. 

\begin{lemma}  \label{lm:6}
Let $Z_{1}, Z_{2} \subset \alpha\cup \beta$ be two $Z$-pieces on a $k$-child $(\alpha,\beta)$ for any $k\in \mathbb{N}$. Then 
\[ Z_{1}^{(\alpha)}\cap Z_{2}^{(\alpha)} \neq \emptyset \Rightarrow Z_{1}= Z_{2} \]
\end{lemma}

\begin{proof} Without loss of generality, $Z_{1}$ starts before $Z_{2}$ along $\alpha$, and terminates before $Z_{2}$ along $\alpha$. We claim that it suffices to show 
\[ Z_{1} \cap \alpha = Z_{2} \cap \alpha; \]
indeed, the arcs of $\alpha$ contained in a $Z$-piece determine uniquely the arcs of $\beta$ comprising the companion arc, and therefore 
\[ Z_{1} \cap \alpha= Z_{2} \cap \alpha \Rightarrow Z_{1} \cap \beta= Z_{2} \cap \beta \]
\[ \Rightarrow Z_{1}= Z_{2}.\]

Identify $Z_{2}$ with the $Z$-piece pictured in Figure $3$; that is, the $\alpha$-arcs within $Z_{2}$ are labeled $x_{1},x_{2},...,x_{6}$, and the $\beta$-arcs are labeled $y_{1},...,y_{6}$ in that order, respectively. 

\textbf{Case 1: $x_{6}(Z_{1})=x_{2}(Z_{2})$.} Note that, in Figure $3$, the concatenation of arcs $y_{5} \ast x_{2}^{-1}$ forms a loop. 

Therefore, if $x_{6}(Z_{1})=x_{2}(Z_{2})$, the concatenation $y_{3} \ast x_{4}$ must also be a loop. By inspection, it is not. \vspace{1 mm}

\textbf{Case 2: $x_{6}(Z_{1})= x_{3}(Z_{2})$.} As seen in Figure $3$, the concatenation $y_{3}\ast x_{4} \ast x_{5}$ is a loop.  

Then if $x_{6}(Z_{1})= x_{3}(Z_{2})$, the concatenation of $y_{6}$ with the next two sub-arcs of $\alpha$ also forms a loop, but this would imply that $y_{6}(Z_{2})$ coincides with the initial point of $x_{1}$, contradicting Remark \ref{rk:2}.  \vspace{1 mm}

\textbf{Case 3: $x_{6}(Z_{1})=x_{4}(Z_{2})$.} Note that the concatenation of arcs $y_{2}^{-1} \ast x_{5}$ forms a loop.  

Then as in the previous $2$ cases, if $x_{6}(Z_{1})=x_{4}(Z_{2})$, the concatenation $y_{3}^{-1} \ast x_{6}$ must be a loop, but as seen in Figure $3$, it is not. \vspace{1 mm}

\textbf{Case 4: $x_{6}(Z_{1})=x_{5}(Z_{2})$.} The concatenation $C$ of $y_{1}^{-1}$ with the forward direction of the next $\alpha$ sub-arc can not be a loop, by Remark \ref{rk:2}. However, $y_{2}^{-1} \ast x_{5}$ is a loop, and therefore $x_{6}(Z_{1})= x_{5}(Z_{2})$ implies that $C$ is a loop.  \vspace{2 mm}

Therefore, if $ Z_{1}^{(\alpha)}\cap Z_{2}^{(\alpha)} \neq \emptyset$, $Z_{1}^{(\alpha)}=Z_{2}^{(\alpha)}$, and thus $Z_{1}=Z_{2}$.

\end{proof}

Since, by Lemma \ref{lm:5}, $Z$-pieces are symmetric with respect to the arcs $a$ and $b$, Lemma \ref{lm:6} immediately implies the following corollary:

\begin{corollary} \label{cor:1}
Let $Z_{1}, Z_{2} \subset \alpha\cup \beta$ be two $Z$-pieces on a $k$-child $(\alpha,\beta)$ for any $k\in \mathbb{N}$. Then 
\[ Z_{1}^{(\beta)}\cap Z_{2}^{(\beta)} \neq \emptyset \Rightarrow Z_{1}= Z_{2}. \]
\end{corollary}
\begin{proof} If, on the filling pair $(\alpha,\beta)$, $Z_{1}$ and $Z_{2}$ share $\beta$-arcs, then by Lemma \ref{lm:5}, on the pair $(\alpha',\beta'):=(\beta,\alpha)$, $Z_{1}$ and $Z_{2}$ share $\alpha'$-arcs; apply Lemma \ref{lm:6}.
\end{proof}

Putting Lemma \ref{lm:6} and Corollary \ref{cor:1} together, we obtain the following important fact, which states that distinct $Z$-pieces can not overlap:

\begin{lemma} \label{lm:7}
Let $Z_{1},Z_{2}$ be $Z$-pieces on a $k$-child $(\alpha,\beta)$ for any $k\in \mathbb{N}$. Then 
$$
(Z_{1},Z_{2})^{\cap} \neq \emptyset \Rightarrow Z_{1} = Z_{2}.
$$
\end{lemma}

\begin{proof} Immediate from \ref{lm:6}, Corollary \ref{cor:1} and the definition of $(Z_{1},Z_{2})^{\cap}$.
\end{proof}

Henceforth, we assume that $g \geq 3$ is odd. The even case is completely analogous, and will therefore follow from the construction of a minimally intersecting filling pair on $S_{4}$ at the end of this section.

Given an odd integer $g \geq 3$, we define the set $L_{g}$ of integer sequences by
\[L_{g}:= \left\{(\alpha_{1}, \alpha_{2},...,\alpha_{\frac{g-1}{2}} ) : \alpha_{1}< \alpha_{2}<...<\alpha_{\frac{g-1}{2}}, \alpha_{i} \leq 4i-3 \right\}.\]

Let $W_{g}$ denote the set of all oriented minimally intersecting filling pairs on $S_{g}$.

\begin{lemma} \label{lm:8}
$|W_{g}| \geq |L_{g}|.$
\end{lemma}

\begin{proof} We show the existence of an injection $f=f_{g}:L_{g} \hookrightarrow W_{g}$. That is, given any integer sequence $(\alpha)=(\alpha_{1},...,\alpha_{\frac{g-1}{2}})$ in $L_{g}$, we will show how to construct a corresponding oriented minimally intersecting filling pair in a unique way. 

Start with the filling pair $(\alpha(1),\beta(1))$ on $S_{1}$, and attach a $Z$-piece along the only intersection point to obtain an oriented minimally intersecting filling pair $(\alpha(3),\beta(3))$ on $S_{3}$. As in the original construction at the beginning of this section, label the $2g-1=5$ intersection points $v_{1},v_{2},...,v_{5}$ by choosing an initial intersection point and counting with respect to the orientation along $\alpha(3)$. Note that $\alpha_{2} \leq 5$; then attach a $Z$-piece by excising a small disk centered around $v_{\alpha_{2}}$ and gluing to obtain an oriented pair $(\alpha(5),\beta(5))$ on $S_{5}$. 

Relabel the $2(5)-1=9$ intersection points as follows: any of the original intersection points $v_{k}$ coming from $\alpha(3) \cup \beta(3)$ with $k <\alpha_{2}$ are labeled $v_{k}$. Intersection points $v_{j}$ coming from $\alpha(3) \cup \beta(3)$ with $j> \alpha_{2}$ are relabeled as $v_{j+4}$. The $5$ intersection points coming from the $Z$-piece are labeled, in order along the arc $a$, $v_{\alpha_{2}},v_{\alpha_{2}+1},v_{\alpha_{2}+2},v_{\alpha_{2}+3}, v_{\alpha_{2}+4}$. This produces a labeling of all $9$ intersection points which respects the orientation of $\alpha(5)$. 

We then glue on an additional $Z$-piece by excising the intersection point $v_{\alpha_{3}}$, obtaining the oriented pair $(\alpha(7),\beta(7))$ on $S_{7}$, and we relabel the now $13$ intersection points as above. 

We repeat this process for each $i \leq \frac{g-1}{2}$, each time excising the intersection point labeled $v_{\alpha_{i}}$ and gluing on another $Z$-piece. This process terminates with an oriented minimally intersecting filling pair in $W_{g}$, which we define to be $f((\alpha))$. 

\begin{remark} \label{rk:3}
It is unclear how many $Z$-pieces the filling pair $f((\alpha))$ contains. However, it necessarily contains at least one $Z$-piece: the one associated to the last index $\alpha_{\frac{g-1}{2}}$. 
\end{remark}

It remains to show that $f$ is injective. Assume $f((\alpha))=f((\alpha'))$. Then we claim that 
\[ \alpha_{\frac{g-1}{2}}= \alpha'_{\frac{g-1}{2}}. \]

Suppose not; then by Remark \ref{rk:3}, the filling pair $f((\alpha))=f((\alpha'))$ contains $2$ distinct $Z$-pieces $Z,Z'$, associated to $\alpha_{\frac{g-1}{2}}$ and $\alpha'_{\frac{g-1}{2}}$, respectively. Without loss of generality, $\alpha_{\frac{g-1}{2}} < \alpha'_{\frac{g-1}{2}}$. By Lemma \ref{lm:7}, $Z$ and $Z'$ can not overlap. 

Therefore, if we remove $Z$ by cutting along the separating curve in $S_{g}$ associated to its boundary and gluing in a disk, we obtain an oriented minimally intersecting filling pair on $S_{g-2}$, which still contains $Z'$. Then since $(\alpha)$ is a monotonically increasing sequence, by the same argument the filling pair obtained by successively cutting along the $Z$-pieces associated to smaller and smaller indices of $(\alpha)$ will still contain $Z'$. Eventually, we conclude that the filling pair $(\alpha(1),\beta(1))$ on $S_{1}$ contains $Z'$, a contradiction. Hence, $\alpha_{\frac{g-1}{2}}=\alpha'_{\frac{g-1}{2}}$. 

Applying the exact same argument inductively, we have $\alpha_{k}= \alpha'_{k}$ for each $k < \frac{g-1}{2}$ as well. Hence $f$ is injective.  
\end{proof}

\begin{corollary} \label{cor:2}
For $g$ odd, 
$$N(g) > \frac{\prod_{k=1}^{\frac{g-1}{2}} 3k+2}{4\cdot (2g-1)^{2} \cdot \left(\frac{g-1}{2}\right)!}.$$

\end{corollary}

\begin{proof} By Lemma \ref{lm:8}, it suffices to show that

$$
|L_{g}|> \frac{\prod_{k=1}^{\frac{g-1}{2}-1} 3k+2}{ \left(\frac{g-1}{2}-1\right)!}.
$$
The lower bound on $N(g)$ follows by dividing by an upper bound on the number of twisting permutations. 

Moreover, we bound $|L_{g}|$ by counting the number of length $\frac{g-1}{2}-1$ integer sequences $(\alpha_{i})_{i}$ such that $\alpha_{i} \leq 4i-3$, and then dividing by the order of the permutation group $\Sigma_{\frac{g-1}{2}-1}$ to account for the fact that the sequences in $L_{g}$ are monotonically increasing. 
\end{proof}
The following lemma then implies the exponential growth of $N(g)$:
\begin{lemma} \label{lm:9}
\[\frac{\prod_{k=1}^{\frac{g-1}{2}-1} 3k+2}{4\cdot (2g-1)^{2} \cdot \left(\frac{g-1}{2}-1\right)!} \sim \frac{3^{g/2}}{g^{2}}.\]
\end{lemma}
\begin{proof} Immediate from the definition of $\sim$ after factoring a $3^{(g-1)/2}$ out from the numerator.
\end{proof}

As mentioned above, the proof of the lower bound for $g$ even, $g>2$ is exactly analogous. Hence to complete the proof of the lower bound, it suffices to construct a minimally intersecting filling pair on $S_{4}$; equivalently, we demonstrate the existence of a filling permutation in $\Sigma_{28}$: 
$$
(1,14,27,18,9,8,17,26,7,4,25,20,13,2,7,10,5,12,15,28,3,10,19,22)\in \Sigma_{28}. 
$$

Recall that the proof of Lemma \ref{lm:8} relied on the assumption that there are no $Z$-pieces in this filling pair. In section $2.4$ below, we show that there are no filling pairs on $S_{2}$ whose complement is connected. Thus, if there was a $Z$-piece within the filling pair on $S_{4}$ associated to the permutation above in $\Sigma_{28}$, we could excise a copy of $Z$ from $S_{4}$ to obtain a filling pair on $S_{2}$ whose complement is connected, a contradiction.

\subsection{Proof of the Upper Bound}

The upper bound from Theorem \ref{thm:1} is obtained by bounding the number of filling permutations from above. 

\begin{lemma} \label{lm:10}
Let $\sigma \in \Sigma_{8g-4}$ be a filling permutation. Then $\sigma= Q^{4g-2}C$, where $C$ is a square root of $Q^{4g-2}\tau$. Conversely, if $C$ is a square root of $Q^{4g-2}\tau$ such that $\sigma= Q^{4g-2}C$ is a parity respecting $(8g-4)$-cycle, then $\sigma$ is a filling permutation.

\end{lemma}

\begin{proof} Suppose $\sigma$ is a filling permutation. Then 
$$
 (Q^{4g-2}\sigma)^{2}= Q^{4g-2}(\sigma Q^{4g-2} \sigma) \text{, and } Q^{4g-2} \tau.
 $$

Conversely, if $C$ is a square root of $Q^{4g-2}\tau$,
$$
 (Q^{4g-2}C)Q^{4g-2}(Q^{4g-2}C)= Q^{4g-2}C^{2}= \tau.
 $$
\end{proof}
Therefore, $N(g)$ is bounded above by the number of square roots of $Q^{4g-2}\tau$; the upper bound of Theorem \ref{thm:1} is obtained by identifying a collection of such square roots $C$ such that $Q^{4g-2}C$ is not an $(8g-4)$-cycle, and subtracting this from the total number of square roots. 

Note that $Q^{4g-2}\tau$ is a composition of disjoint transpositions:

\[ Q^{4g-2}\tau=  (1,4g+1)(2,4g+2)(3,4g+3)...(4g-4,8g-4)(4g-3,4g-1)(4g-2,4g).\] 

The square roots of $Q^{4g-2}\tau$ are obtained by partitioning the above transpositions into pairs and ``interleaving'' the pairs. For instance, one such root is obtained by pairing off $(1,4g+1)$ with $(2,4g+2)$, and in general for each of $i\leq 4g-3$, pairing off the $i^{th}$ transposition above with the $(i+1)^{st}$, and converting each such pair into a single $4$-cycle by interleaving the $4$ numbers :
\[ (1,2,4g+1,4g+2)(3,4,4g+3,4g+4)...(4g-3,4g-2,4g-1,4g).\]

Note that for a particular pair of transpositions, there are two $4$-cycles we can obtain by interleaving. Furthermore, because $Q^{4g-2}$ is parity respecting and sends even numbers to even numbers, the only way $Q^{4g-2}C$ can be a parity respecting $(8g-4)$-cycle is if it sends even numbers to odd ones. Therefore there are exactly $2^{2g-1}(2g-1)!$ square roots $C$ of $Q^{4g-2}\tau$ such that $Q^{4g-2}C$ is parity respecting. 

We remark that the fact that $2^{2g-1}(2g-1)!$ is an upper bound for $N(g)$ is immediate without the help of Lemma \ref{lm:10} (label the intersection points from $1$ to $2g-1$ according to their order along $\alpha$ and comparing this labeling to an analogous one obtained by ordering along $\beta$ produces a permutation in $\Sigma_{2g-1}$; specifying the sign of each intersection accounts for the factor of $2^{2g-1}$). 

However, using Lemma \ref{lm:10}  and the structure of the square roots of $Q^{4g-2}\tau$, we can obtain strictly smaller lower bounds. In particular, consider the following $2^{2g-2}(2g-1)(2g-3)!$ permutations:

Pair the transposition $(1,4g+1)$ with an arbitrary transposition $(k, j)$ of even numbers in the disjoint cycle notation for $Q^{4g-2}\tau$; there are $2\cdot (2g-1)$ ways of doing this. If $C$ is any square root of $Q^{4g-2}\tau$ with $(1,k,4g+1,j)$ as one of its $4$-cycles, then $Q^{4g-2}C$ maps $1$ to $(k+4g-2) \pmod{8g-4}$ (or $8g-4$ if $k=4g-2$), the number representing the inverse direction of the arc corresponding to $k$. Note that by inspection of $Q^{4g-2}\tau$, 
$$ (k+ 4g-2) \neq  j  \pmod{8g-4}, $$
and therefore we are free to pair the transposition containing $(k+4g-2)  \pmod{8g-4}$ with the transposition containing $4g-1$, and we choose the interleaving of these transpositions so that the resulting $4$ cycle sends $(k+4g-2) \pmod{8g-4}$ to $4g-1$. 

Then we can obtain a parity respecting square root $C$ by pairing off the remaining transpositions in any way, evens with odds. However, no square root $C$ obtained in this fashion has the property that $Q^{4g-2}C$ is an $(8g-4)$-cycle, because by construction,
$$( Q^{4g-2}C)^{2}(1)= 1.$$

Therefore, the number of admissible square roots is bounded above by 
$$2^{2g-1}- 2^{2g-2}(2g-1)(2g-3)! = 2^{2g-2}\frac{(4g-5)(2g-1)!}{2g-2}. $$
Note that the number of twisting permutations is bounded below by $2g-1$, since $\kappa_{g}$ has order $2g-1$. This completes the proof of the upper bound of Theorem \ref{thm:1}. $\Box$

\subsection{Genus $2$}
In this section, we show that there are no filling pairs intersecting $3$ times on $S_{2}$. For a filling pair with $4$ intersections on $S_{2}$, see page $41$ of Farb-Margalit \cite{Far-Mar}. 

\begin{theorem} \label{thm:4}
Suppose $(\alpha,\beta)$ is a filling pair on $S_{2}$. Then $i(\alpha,\beta)\geq 4$. 

\end{theorem}

\begin{proof} Suppose $i(\alpha,\beta)=3$. We can schematically represent $\alpha \cup \beta$ as a partially directed graph $G=(V,E)$ with $6$ vertices, $3$ undirected edges and $6$ directed edges as follows. To construct $G$, begin with two disjoint $3$-cycles, one which is oriented clockwise and the other counter-clockwise. Label the vertices on each cycle from $1$ to $3$ in cyclic order. The two directed $3$-cycles represent $\alpha$ on each of its two sides, and the three intersection points between $\alpha$ and $\beta$ correspond to the $3$ vertices on each cycle. 

The remaining $3$ undirected edges of $G$ correspond to the $3$ arcs of $\beta$. If $(\alpha, \beta)$ fill, their complement is a single disk, and this corresponds to the existence of a certain kind of Eulerian cycle. Concretely, there is a cycle $c$ on $G$ beginning with the oriented edge on the clockwise cycle whose initial point is vertex $1$, satisfying the following:

\begin{enumerate}
\item The edges along $c$ alternate between being directed and undirected;
\item every directed edge appears along $c$ exactly once;
\item every undirected edge appears along $c$ exactly once in each direction.
\end{enumerate}

As demonstrated below in Figure $5$, no undirected edge can connect vertices on the same $3$-cycle, or else $\alpha$ and $\beta$ can not be in minimal position. 

Therefore, without loss of generality we can assume that the first undirected edge connects vertex $1$ on the clockwise cycle to vertex $3$ on the counter-clockwise cycle. This will then determine $G$, because vertex $3$ on the clockwise cycle can not then be connected to vertex $1$ on the counter-clockwise cycle, or else $\beta$ will close up without ever passing through vertex $2$. 

Therefore, $G$ must be as pictured in Figure $6$. Note then that $G$ can not satisfy the Eulerian condition described above; starting at vertex $1$ and traveling along the clockwise cycle, one arrives at vertex $2$. Then the undirected edge incident to vertex $2$ on the clockwise cycle is connected to vertex $1$ on the counter-clockwise cycle at its other end. Then we must travel counter-clockwise to vertex $3$ on the counter-clockwise cycle, and finally the undirected edge at this vertex connects back to vertex $1$ on the clockwise cycle. Therefore, no single cycle can satisfy the three necessary requirements. 

\begin{figure}[H]
\centering
	\includegraphics[width=3in]{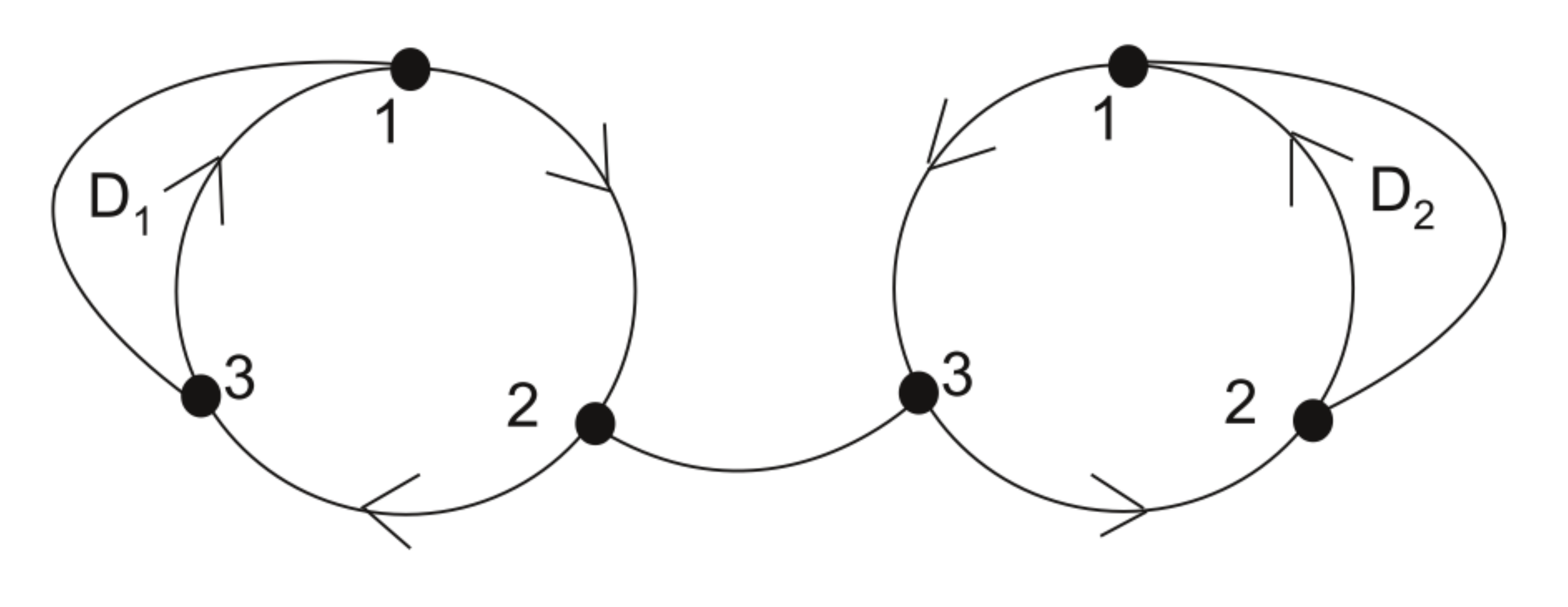}
\caption{If one of the undirected edges connects vertices on the same cycle, there must be at least two undirected edges with this property. This will imply the existence of two bigons labeled $D_{1},D_{2}$ above in the complement of $\alpha\cup \beta$.}
\end{figure}

\begin{figure}[H]
\centering
	\includegraphics[width=3in]{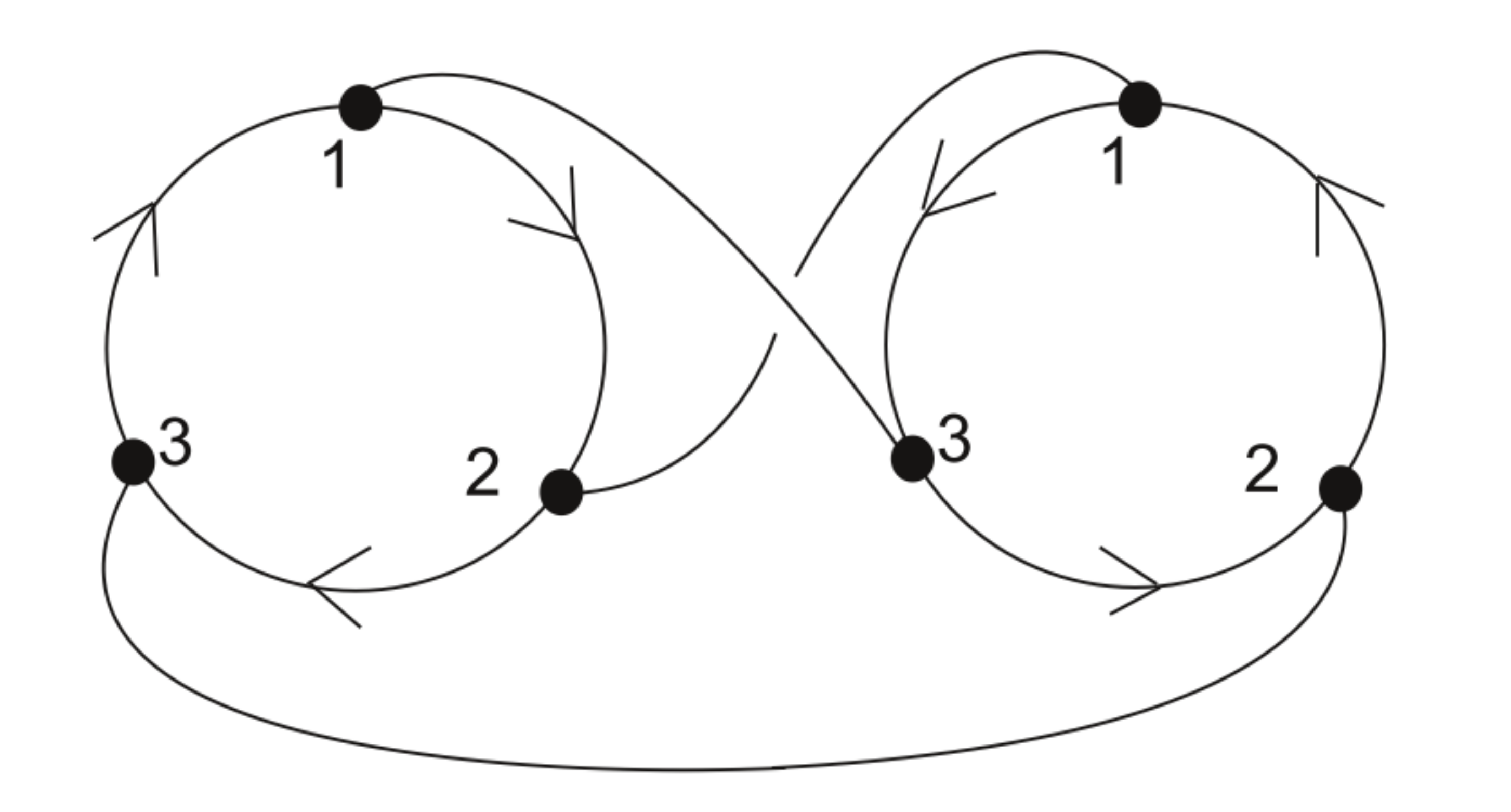}
\caption{When undirected edges connect vertices on opposite cycles, $G$ can not satisfy the Eulerian condition described above.}
\end{figure} 

\end{proof}

\section{ Minimally Intersecting Filling Pairs as Combinatorial Optimizers}

In this section, we consider a potential way of measuring the combinatorial efficiency of a filling pair. Namely, given a filling pair $(\alpha,\beta)$, let $T_{k}(\alpha,\beta)\in \mathbb{N}$ be the number of simple closed curves (up to isotopy) which intersect the union $\alpha \cup \beta$ no more than $k$ times. We expect for an ``optimal'' filling pair to have large values of $T_{k}$ for small values of $k$; in other words, there should be many simple closed curves which only intersect the union a small number of times. 

Specifically, we prove Theorem \ref{thm:2} which shows that minimally intersecting filling pairs optimize the value of $T_{1}$:

\textbf{Theorem 1.2. } 
\textit{Let $\left\{\alpha,\beta\right\}$ be a filling pair on $S_{g}, g>2$ and define $T_{1}(\alpha,\beta) \in \mathbb{N}$ to be the number of simple closed curves intersecting $\alpha \cup \beta$ only once. Then $T_{1}(\alpha,\beta)\leq 4g-2$, with equality if $(\alpha,\beta)$ is minimally intersecting. } 
\vspace{2 mm}

\begin{proof}
Suppose first that $(\alpha,\beta)$ is a minimally intersecting filling pair. Then $T_{1}(\alpha,\beta)= 4g-2$, because for each of the $2g-1$ sub-arcs $\alpha_{k}$ of $\alpha$ (resp. $\beta$), there is a single simple closed curve which intersects $\alpha_{k}$ and no other sub-arc. This simple closed curve corresponds to the arc in the $(8g-4)$-gon obtained by cutting along $\alpha \cup \beta$ which connects the two edges which project down to the sub-arc $\alpha_{k}$. 

Now suppose that $(\alpha,\beta)$ do not intersect minimally, so that $S_{g}\setminus (\alpha \cup \beta)$ has multiple connected components. The number of components is equal to 
\[ i(\alpha,\beta)- 2g+2; \]
thus if we cut along $\alpha \cup \beta$ we obtain a disjoint union of $i(\alpha,\beta)-2g+2$ even-sided polygons $P_{1},...,P_{i(\alpha,\beta)-2g+2}$, such that the total number of edges over all polygons is $4\cdot i(\alpha,\beta)$. 
$T_{1}(\alpha,\beta)$ is then equal to the number of pairs of edges which project to the same arc in $S_{g}$, and which belong to the same polygon. For each $j$, there must be at least one edge of $P_{j}$ which has the property that the inverse edge is not on $P_{j}$ (or else $P_{j}$ will be disconnected from all of the other polygons after glued up). Thus, for each $j$ there must be at least $2$ edges of $P_{j}$ whose inverse edges belong to other polygons, since $P_{j}$ is even-sided for each $j$. 

Suppose there exists $j \in \left\{1,...,i(\alpha,\beta)-2g+2\right\}$ such that there are exactly two edges of $P_{j}$ whose inverse edge is not on $P_{j}$. Then one must project to a sub-arc $\alpha_{k}$ of $\alpha$, and the other to a sub-arc $\beta_{h}$ of $\beta$. Consider the arc $\beta_{l}$ which immediately precedes $\alpha_{k}$ along $P_{j}$, and assume $\beta_{l} \neq \beta_{h}$ (otherwise, the arc immediately preceding $\beta_{h}$ can not be $\alpha_{k}$, so in this case we just switch the roles of $\beta_{h}$ and $\alpha_{k}$). 

By assumption, $(\alpha, \beta)$ is not minimally intersecting, but there is still a map $M$ acting on the edges of $P_{1}\sqcup...\sqcup P_{i(\alpha,\beta)-2g+2}$ playing the same role as the permutation $C$ discussed in section $2.3$ and in the proof of Lemma \ref{lm:2} in section $2.1$. Namely, send an edge $e$  to the edge $e'$ immediately following it in clockwise order along some $P_{i}$, and then send $e'$ to the edge which projects down to the same arc as $e'$ does (i.e., the inverse edge for $e'$). 
 
This map $M$ has order $4$, because it corresponds to a ``rotation'' by $\pi/2$ about an intersection point. Note that $M(\beta_{l})= \alpha_{k}^{-1}$, which by assumption is not  on $P_{j}$. Then $M(\alpha_{k}^{-1})= \beta_{l+1}^{-1}$ (by $\beta_{l+1}$, we simply mean the sub-arc immediately following $\beta_{l}$ in the direction of $\beta_{l}$; see Figure $7$).

\begin{figure}[t]
\centering
	\includegraphics[width=3.5in]{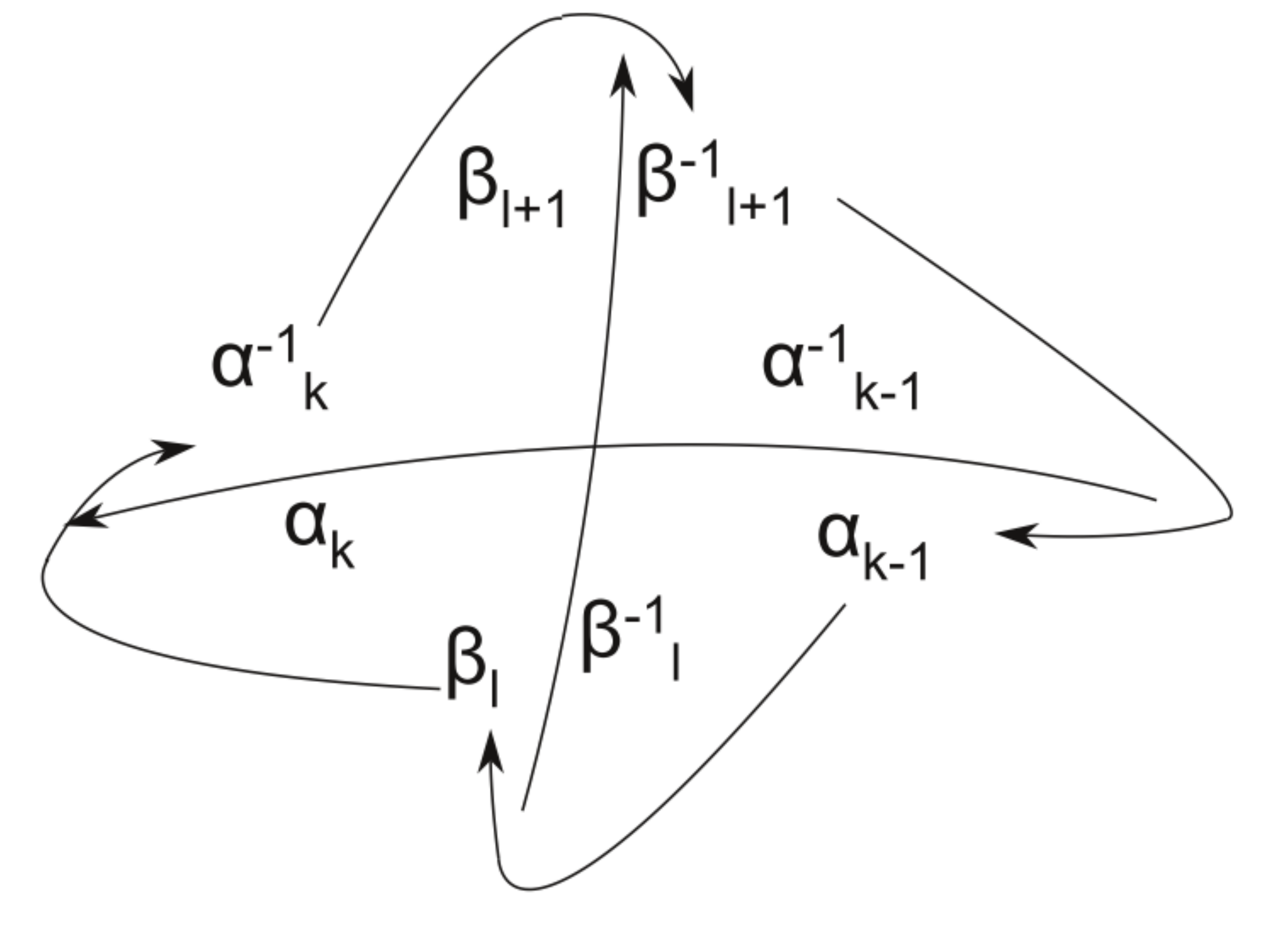}
\caption{Schematic of an intersection point of $\alpha \cup \beta$. Arrows indicate the action of $M$ on each oriented edge.}
\end{figure}

We claim that $\beta_{l+1}= \beta_{h}^{-1}$, for assume not. Then $M(\alpha_{k}^{-1})= \beta_{l+1}^{-1}$ is not on $P_{j}$. $\beta_{l+1}^{-1}$ is immediately followed by $\alpha^{-1}_{k-1}$.  Thus $M(\beta_{l+1}^{-1})= \alpha_{k-1}$, which is also not on $P_{j}$, because otherwise both $\alpha_{k}$ and $\alpha_{k-1}$ would have the property that their inverse edges are not on $P_{j}$. However, $M$ must have order $4$, which implies that $\beta_{l}^{-1}$ immediately follows $\alpha_{k-1}$ (see Figure $7$), and that in particular, the inverse of $\beta_{l}$ is not on $P_{j}$, which contradicts our choice of $\beta_{l}$. 

Therefore, $\beta_{l+1}= \beta_{h}^{-1}$. This in particular implies that $\alpha_{k}^{-1}$ immediately precedes $\beta_{h}^{-1}$ on $P_{m}$ for some $m$. We claim that $\beta_{h}$ must immediately follow $\alpha_{k}$ on $P_{j}$ (see Figure $8$). If not, then let $\beta_{q}$ denote the edge immediately preceding $\alpha_{k}^{-1}$ on $P_{m}$. Then $M(\beta_{q})= \alpha_{k}$ on $P_{j}$, and $M(\alpha_{k})= \beta_{r}$ for some $r$, also on $P_{j}$. Thus $M(\beta_{r})= \alpha_{k+1}^{-1}$, also on $P_{j}$, and since $M$ has order $4$, $\beta_{q}^{-1}$ immediately follows $\alpha_{k+1}^{-1}$ on $P_{j}$, and therefore $\beta_{q}^{-1}= \beta_{h}$. But this is a contradiction, because $\beta_{q}$ and $\beta_{h}^{-1}$ represent distinct edges on $P_{m}$; see Figure $8$.

\begin{figure}[t]
\centering
	\includegraphics[width=3.5in]{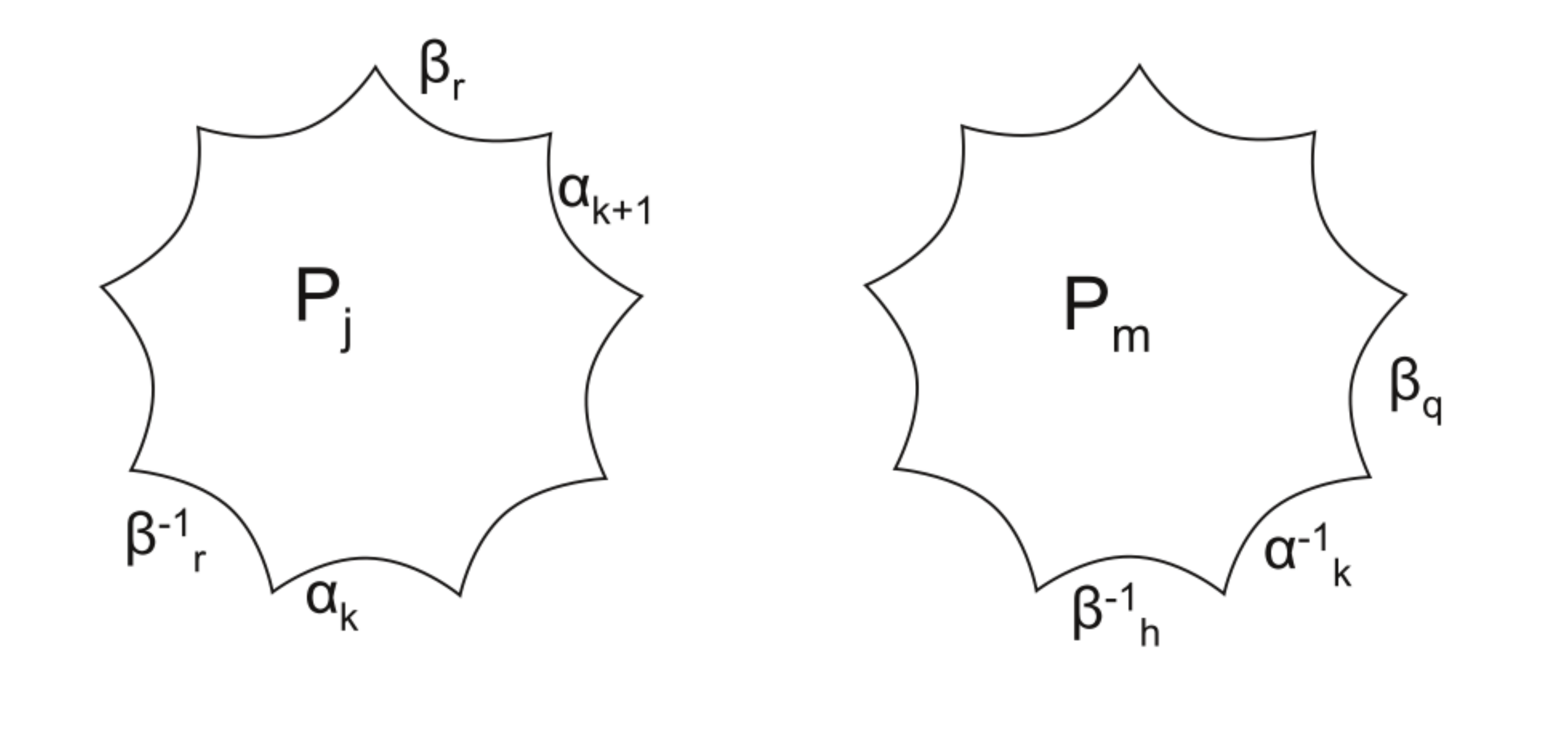}
\caption{Using the fact that $M$ has order $4$, we conclude that $\beta_{q}^{-1}=\beta_{h}$, but this is impossible because $\beta_{q}$ and $\beta_{h}^{-1}$ are distinct edges of $P_{m}$ by assumption.}
\end{figure}

This implies that a sufficiently small neighborhood in $S_{g}$ around the vertex on $P_{j}$ separating $\alpha_{k}$ and $\beta_{j}$ contains only two edges of $\alpha \cup \beta$. This is a contradiction, because $\alpha \cup \beta$ is $4$-valent. Since $j$ was arbitrary, we conclude that each polygon $P_{1},...,P_{i(\alpha,\beta)-2g+2}$ must have at least 4 edges whose inverses are not on the same polygon. This leaves at most 
\[ 4i(\alpha,\beta)- 4(i(\alpha,\beta)-2g+2)= 8g-4\]
edges whose inverse can potentially be on the same polygon, and therefore 
\[ T_{1}(\alpha,\beta)\leq 4g-2.\]
\end{proof}

The example of a filling pair on $S_{2}$ shown in Figure $4$ proves that one can not push this further, in the sense that there exists filling pairs with complementary polygons $P$ containing only $4$ edges whose inverse edge is not on $P$. However, one can say more in the case that there exists a complementary region with a number of edges not dividing $4$:

\begin{corollary} \label{cor:3}
Let $(\alpha,\beta)$ be a filling pair having the property that one complementary region $P$ has a number of sides which is not divisible by $4$. Then $T_{1}(\alpha,\beta)\leq 4g-4$. 

\end{corollary} 

\begin{proof} The proof of Theorem \ref{thm:2} implies that each complementary region has at least $4$ sides whose inverse edge is not on the same polygon. Furthermore, if there exists a polygon $P$ in the complement of $\alpha \cup \beta$ having a number of edges not divisible by $4$, it follows that the number of edges of $P$ projecting to $\alpha$ is odd, and similarly for $\beta$. Thus, if $\alpha_{k},\alpha_{j}$ are the $2$ $\alpha$-edges of $P$ guaranteed by Theorem \ref{thm:2} to have the property that their inverse edges are not on $P$, the remaining number of $\alpha$-edges on $P$ is still odd, and therefore they can not all glue together. Hence, there are at least $6$ edges of $P$ whose inverse edges are on other polygons, and therefore $T_{1}(\alpha,\beta)\leq 4g-4$.

\end{proof}

It would be interesting to analyze the statistics of $T_{k}(\alpha,\beta)$ as $k$ goes to $\infty$. We pose the following question:

\begin{question} What is the minimal $k$ such that $T_{k}$ is not maximized by a minimally intersecting filling pair?
\end{question}

\section{Minimal Length Filling Pairs on Hyperbolic Surfaces}
In this section, we prove Theorem \ref{thm:3}. In all that follows, let $l_{h}$ denote  length in $\mathbb{H}^{2}$, the hyperbolic plane. As in the introduction, define $\mathcal{M}_{g}$ to be the \textit{Moduli Space} of $S_{g}$, the space of all complete hyperbolic metrics on $S_{g}$ up to isometry. Then  $\mathcal{F}_{g}:\mathcal{M}(S_{g})\rightarrow \mathbb{R}$ is the function which outputs the length of the shortest minimally intersecting filling pair on a hyperbolic surface $\sigma$, where the length of a filling pair is the sum of the two individual lengths of the unique geodesic representatives of the simple closed curves in the pair. 

In recalling the definition of a topological Morse function, we follow \cite{Schal2}:

Let $f: \mathcal{M}_{g} \rightarrow \mathbb{R}$ be a continuous function. Then $x \in \mathcal{M}_{g}$ is called an \textit{ordinary point} of $f$ if there exists an open neighborhood $U$ of $x$, a chart $\phi: U \rightarrow \mathbb{R}^{6g-6}$ inducing real-valued coordinates $y_{1},...,y_{6g-6}$, and an integer $j$, $1 \leq j \leq 6g-6$ such that for any $z \in U$, 

\[ f(z) = y_{j}.\]

Otherwise, $x$ is a \textit{critical point}. If $x$ is a critical point, $x$ is called a \textit{non-degenerate critical point} of $f$ if there exists an open neighborhood $U$ of $x$, a chart $\phi: U \rightarrow \mathbb{R}^{6g-6}$ inducing real-valued coordinates $y_{1},...,y_{6g-6}$, and an integer $j$, $1\leq j \leq 6g-6$ such that for any $z \in U$, 

\[ f(z)-f(x) =\sum_{i=1}^{j}y_{i}^{2}- \sum_{i=j+1}^{6g-6}y_{i}^{2}.\]

$j$ is called the \textit{index} of the non-degenerate critical point $x$. For our purposes, a \textit{topological Morse function} is a real-valued continuous function $f$ on $\mathcal{M}_{g}$ such that $f$ has only finitely many critical points, all of which are non-degenerate.

Theorem \ref{thm:3} says that $\mathcal{F}_{g}$ is a topological Morse function, and that the number of critical points of index $0$ (i.e., the global minima) grow at least exponentially in genus:

\textbf{Theorem 1.3. }
Let $g\geq 3$. \textit{$\mathcal{F}_{g}$ is proper and a topological Morse function. For any $\sigma \in \mathcal{M}_{g}$, 
\[ \mathcal{F}_{g}(\sigma)\geq \frac{m_{g}}{2},\]
where 
\[m_{g}= (8g-4)\cdot \cosh^{-1}\left( 2 \left[\cos\left(\frac{2\pi}{8g-4}\right)+\frac{1}{2}\right] \right)  \]
denotes the perimeter of a regular, right-angled $(8g-4)$-gon. Furthermore, define 
\[ \mathcal{B}_{g}:= \left\{\sigma \in \mathcal{M}_{g} : \mathcal{F}_{g}(\sigma) = \frac{m_{g}}{2}\right\};\]
then $\mathcal{B}_{g}$ is finite and grows exponentially in $g$. If $\sigma\in \mathcal{B}_{g}$, the injectivity radius of $\sigma$ is at least  $\frac{1}{2} \cosh^{-1}\left(\frac{9}{\sqrt{73}}\right)$ }.

\begin{proof}
 Let $C \subset \mathbb{R}$ be a compact subset. Then we claim that there exists $\epsilon>0$ such that $\mathcal{F}_{g}^{-1}(C)$ is contained in $\mathcal{M}^{\epsilon}_{g}$, the $\epsilon$-\textit{thick part} of moduli space, defined to be the subset of $\mathcal{M}_{g}$ consisting of hyperbolic surfaces with injectivity radius at least $\epsilon$. To see this, note that this is implied by the statement $\mathcal{F}_{g}\rightarrow \infty$ as injectivity radius approaches $0$.

\begin{lemma} \label{lm:11}
For $k \in \mathbb{N}$, there exists $\epsilon=\epsilon(k)$ such that any hyperbolic surface $\sigma$ with injectivity radius $<\epsilon$ satisfies $\mathcal{F}_{g}(\sigma)>k$. 
\end{lemma}
\begin{proof} By the collar lemma (see \cite{Bus}), there exists a function $c: \mathbb{R}\rightarrow \mathbb{R}$ such that $\lim_{x\rightarrow 0}c(x)= \infty$ such that if $\sigma$ is a hyperbolic surface with a closed geodesic $\gamma$ of length $<2\epsilon$, there exists an embedded tubular neighborhood around $\gamma$ whose width is at least $c(\epsilon)$. Since any filling pair must intersect $\gamma$ at least once, the result follows.  
\end{proof}

Therefore $\mathcal{F}_{g}^{-1}(C)$ is contained in the thick part of $\mathcal{M}_{g}$, which is compact by a result of Mumford \cite{Mum}; then continuity of $\mathcal{F}_{g}^{-1}$ implies properness.

To show that $\mathcal{F}_{g}$ is a topological Morse function, we use the following criterion, due to Akrout \cite{Akro}:

Let $(l_{u})_{u\in \mathcal{C}}$ be a collection of real-valued smooth functions defined over $\mathcal{M}_{g}$ and indexed by some countable set $\mathcal{C}$. Given $x \in \mathcal{M}_{g}$ and $\lambda \in \mathbb{R}$, define 
\[ \mathcal{C}_{x}^{\leq \lambda}:= \left\{u \in \mathcal{C}: l_{u}(x)\leq \lambda \right\}.\]
Then the function $\rho:= \inf_{u\in \mathcal{C}}l_{u}$ is called a \textit{generalized systole function} if for each $x \in \mathcal{M}_{g}$ and each $\lambda \in \mathbb{R}$, there exists a neighborhood $W$ of $x$ such that $\bigcup_{q\in W}\mathcal{C}_{q}^{\leq \lambda}$ is a finite subset of $\mathcal{C}$. 

Akrout shows that any generalized systole function is a topological Morse function.  Define $\mathcal{C}$ to be the collection of minimally intersecting filling pairs (not taken up to homeomorphism), and for each such pair $u\in \mathcal{C}$, let $l_{u}(\sigma)$ denote the length of the filling pair $u$ on the hyperbolic surface $\sigma$. Then $\mathcal{F}_{g}= \inf_{u \in \mathcal{C}}l_{u}$.

Fix $x\in \mathcal{M}_{g}$ and $\lambda >0$; then by the collar lemma, only finitely many pairs in $\mathcal{C}$ can have length less than $\lambda$ on $x$. Let $s$ be (one of) the shortest simple closed geodesics on $x$, and let $l_{s}(x)$ denote the length of $s$ on $x$. Let $U$ be a sufficiently small neighborhood around $x$ so that for any $y\in U$, $|l_{s}(y)-l_{s}(x)|<\epsilon$, for some small $\epsilon$. 

Thus any hyperbolic surface in $U$ has the property that $s$ has length at most $l_{s}(x)+\epsilon$, so by another application of the collar lemma, there can only be finitely many filling pairs having the property that there exists a point in $U$ on which the pair has length at most $\lambda$. Therefore $\mathcal{F}_{g}$ is a topological Morse function.

Next, we identify the minima $\mathcal{B}_{g}$ of $\mathcal{F}_{g}$. Recall that given a minimally intersecting filling pair $(\alpha,\beta)$ on a hyperbolic surface $\sigma$, cutting along $\alpha \cup \beta$ produces a single hyperbolic $(8g-4)$-gon $P_{\sigma}$. 

Fix $n \in \mathbb{N}$, $n\geq 3$, and let $\lambda>0$ be such that there exists a hyperbolic $n$-gon with area $\lambda$ (that is, $\lambda$ is less than the area of an ideal $n$-gon). Then by a result of Bezdek \cite{Bez}, the hyperbolic $n$-gon enclosing an area of $\lambda$ with the smallest perimeter is the regular $n$-gon with area $\lambda$. Therefore, the perimeter of $P$ is at least as large as the perimeter of a regular $(8g-4)$-gon enclosing an area of $2\pi(2g-2)$ (the area of $\sigma$, by Gauss-Bonnet). The regular $(8g-4)$-gon with area $2\pi(2g-2)$ is right-angled. The perimeter of such a polygon is $m_{g}$, and its perimeter is twice the length of the filling pair.

Now we demonstrate a uniform lower bound on the injectivity radius for any $\sigma \in \mathcal{B}_{g}$. We identify a simple closed geodesic $\gamma$ on $\sigma$ with its lift $\tilde{\gamma}$ in the right-angled regular $(8g-4)$-gon $P_{\sigma}$ that one obtains by cutting along the minimally intersecting filling pair $(\alpha,\beta)$ of minimal length $m_{g}/2$ on $\sigma$; $\tilde{\gamma}$ is a disjoint collection of geodesic arcs connecting edges of $P_{\sigma}$. 

Assume first that every arc in $\tilde{\gamma}$ connects adjacent edges of $P_{\sigma}$. Since $\gamma, \alpha,\beta$ are all geodesics, they are in pairwise minimal position, and therefore no two of them form a bigon on the surface. 

If $\tilde{\gamma}$ contains some arc $x$ with one endpoint on some edge $a$ of $P_{\sigma}$ belonging to $\alpha$, and the other endpoint on an adjacent edge $b$ belonging to $\beta$, there must be some arc $y$ of $\tilde{\gamma}$ with an endpoint on the edge of $P_{\sigma}$ corresponding to $b^{-1}$. By the previous paragraph, the other endpoint of $y$ can not be located on the edge $c$ belonging to $\alpha$ immediately following $a$ (see Figure $9$). 

\begin{figure}[t]
\centering
	\includegraphics[width=3.5in]{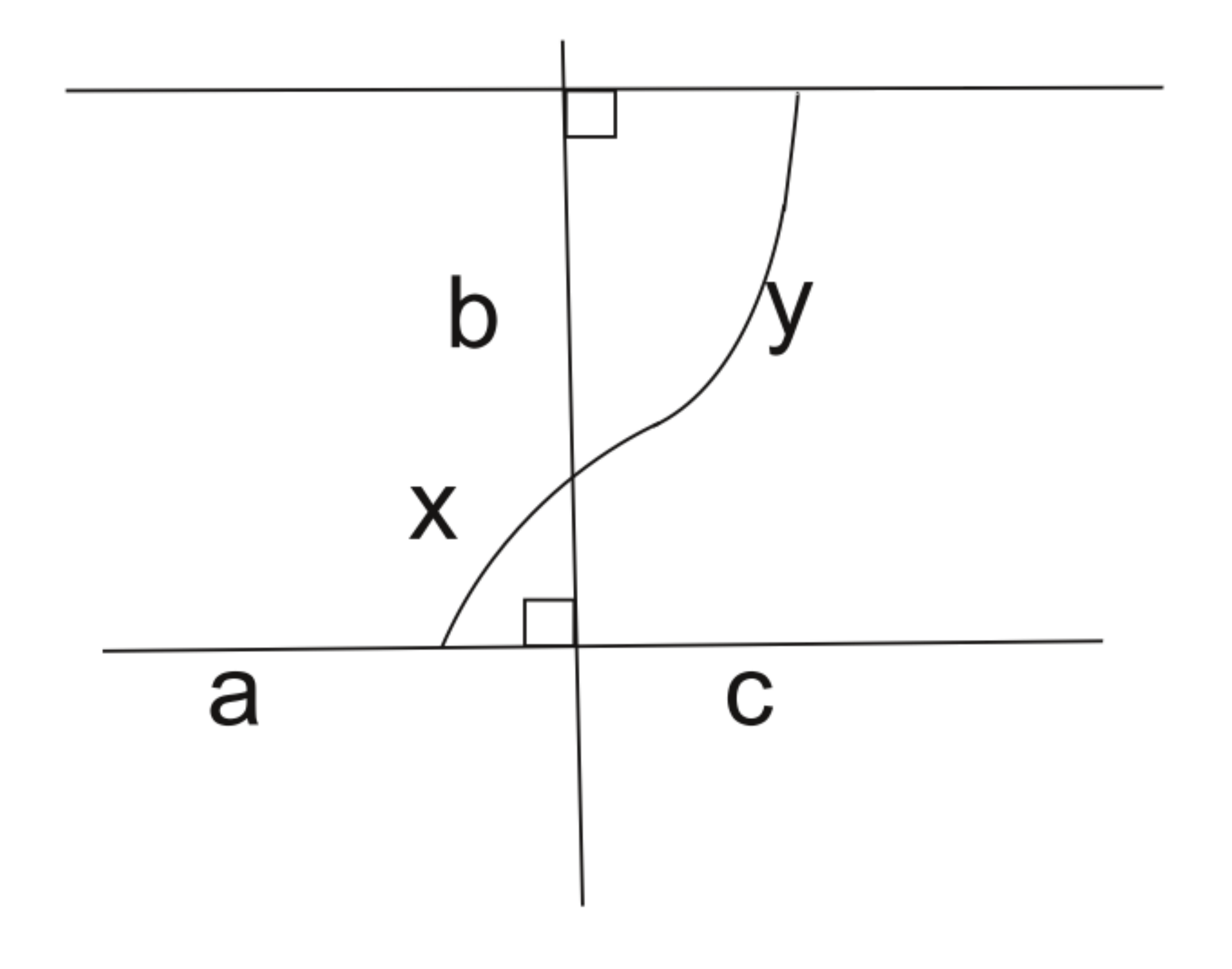}
\caption{If the right endpoint of $y$ is on $c$, $\gamma$ and $\alpha$ can not be in minimal position. }
\end{figure}

But then the arc of $\gamma$ obtained by concatenating $x$ and $y$ is longer than the edge $b$, because $P_{\sigma}$ is right-angled. All edges of $P_{\sigma}$ have length 
\[ \frac{m_{g}}{8g-4}= \cosh^{-1}\left( 2 \left[\cos\left(\frac{2\pi}{8g-4}\right)+\frac{1}{2}\right] \right)  \]

For $g\geq 3$, this is an increasing function of $g$, and therefore 
\[ \frac{m_{g}}{8g-4} > \frac{m_{3}}{20} =  \cosh^{-1}\left(2 \left[\frac{1}{2}+ \sqrt{ \frac{5}{8}+\frac{\sqrt{5}}{8}} \right] \right) > \cosh^{-1}\left(\frac{9}{\sqrt{73}}\right).\]

Therefore, we can assume that there exists an arc $r$ in $\tilde{\gamma}$ whose endpoints are not on adjacent edges of $P_{\sigma}$. Let $e$ be an edge of $P_{\sigma}$ containing an endpoint of $r$. Let $u,v$ denote the endpoints of the two edges adjacent to $e$, which are not shared endpoints with $e$, and let $x$ denote the geodesic segment in $P_{\sigma}$ orthogonal to both $e$ and the geodesic connecting $u$ to $v$ (see Figure $10$). 

\begin{figure}[t]
\centering
	\includegraphics[width=3.5in]{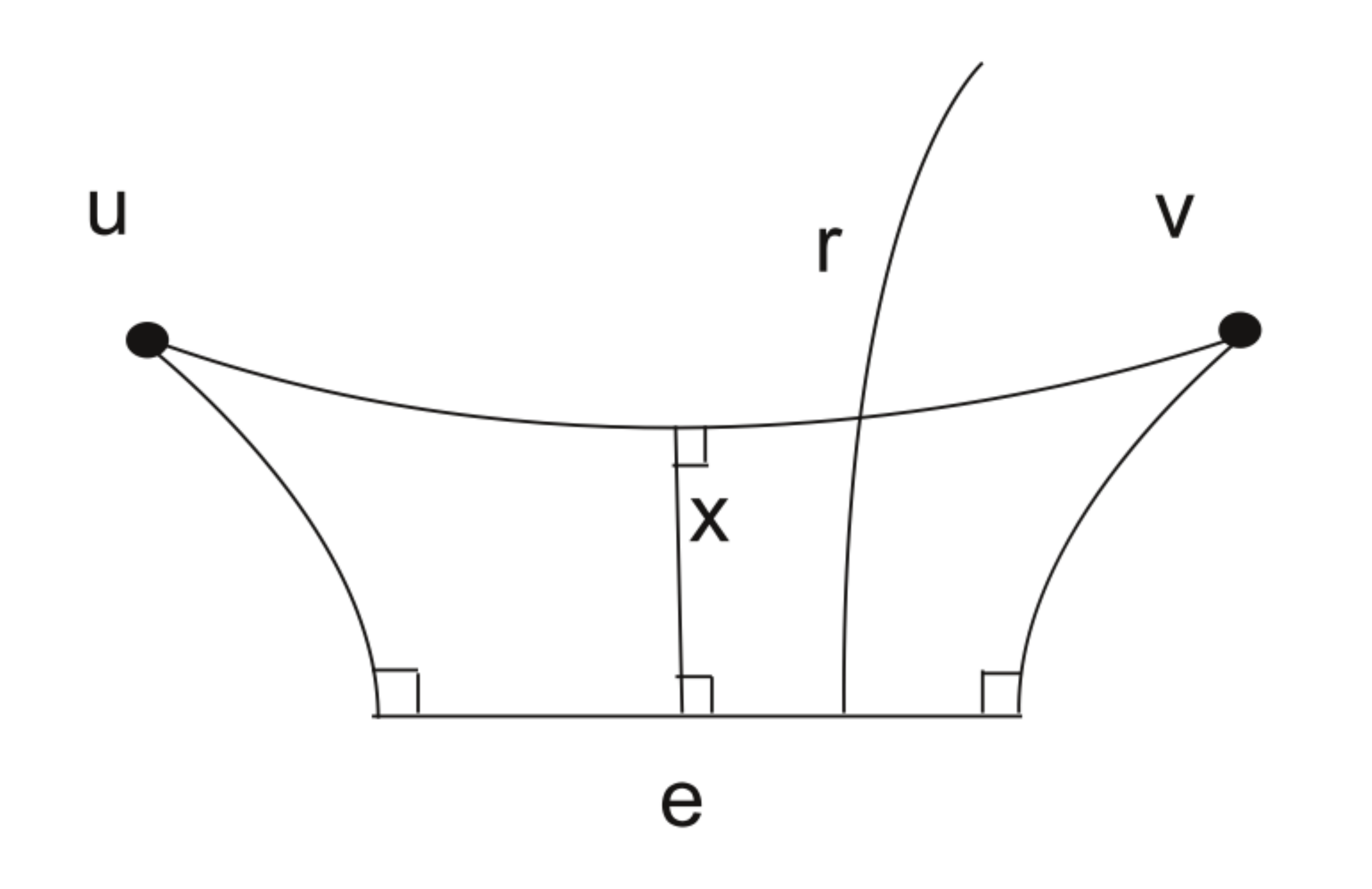}
\caption{$r$ must cross the geodesic connecting $u$ to $v$.}
\end{figure}

We claim that the geodesic segment $g$ connecting $u$ to $v$ must seperate $e$ and its two neighboring edges from the rest of $P_{\sigma}$. Indeed, if any part of $\partial P_{\sigma}$ crosses $g$, then $g$ is not contained in $P_{\sigma}$, which contradicts the fact that $P_{\sigma}$ is convex.   Thus, since by assumption the other endpoint of $r$ is not on an edge adjacent to $e$, $r$ must cross $g$, and therefore its length is at least the length of $x$.  

Computing the length of $x$, which we denote by $\lambda(g)$, is an exercise in elementary hyperbolic geometry- see Theorem $2.3.1$ on page $38$ of \cite{Bus}:

\[ \lambda(g) = \cosh^{-1}\left( \frac{1+ 2\cos\left(\frac{\pi}{2-4g} \right)}{ \sqrt{4 \cos\left(\frac{\pi}{2-4g} \right) \left( 1 +\cos\left(\frac{\pi}{2-4g} \right) \right) + \frac{1}{\left(1+ 2\cos\left(\frac{\pi}{2-4g} \right)\right)^{2}}    } } \right).\]

This is a decreasing function of $g$, and 
\[\lim_{g\rightarrow \infty} \lambda(g) = \cosh^{-1}\left(\frac{9}{\sqrt{73}} \right).\]

This proves that any simple closed geodesic on $\sigma$ has length at least $\cosh^{-1}\left(\frac{9}{\sqrt{73}} \right)$, and therefore the injectivity radius is at least half of this quantity. 

Finally, we adress the growth of $|\mathcal{B}_{g}|$. Properness of $\mathcal{F}_{g}$ implies that all minima are contained in a compact subset of $\mathcal{M}_{g}$, and the fact that $\mathcal{F}_{g}$ is a topological Morse function implies that the minima are isolated. Therefore, $|\mathcal{B}_{g}|<\infty$.

It remains to show that $\mathcal{B}_{g}$ grows exponentially as a function of $g$. Each $\mbox{Mod}(S_{g})$ orbit of minimally intersecting filling pairs can be associated to an element of $\mathcal{B}_{g}$ by simply gluing the edges of a regular right-angled $(8g-4)$-gon together in accordance with the chosen orbit. However, this association need not be injective;  a priori, it may be the case that many minimal length, minimally intersecting filling pairs occur on the same hyperbolic surface. 

Suppose $\sigma$ is a hyperbolic surface admitting $r$ minimal length, minimally intersecting filling pairs. Thus there are $r$ immersed right-angled $(8g-4)$-gons on $\sigma$, and developing $\sigma$ into $\mathbb{H}^{2}$, these polygons lift to $r$ distinct tilings $T_{1},...,T_{r}$ of $\mathbb{H}^{2}$ by regular, right-angled $(8g-4)$-gons.

Let $\Gamma$ denote the holonomy of this developing map; i.e., $\Gamma$ is the discrete subgroup of $PSL_{2}(\mathbb{R})$ identified with the fundamental group of $\sigma$. Recall that the \textit{commensurator} of $\Gamma$, denoted $\mbox{comm}(\Gamma)$, is the subgroup of $\mbox{Isom}(\mathbb{H}^{2})$ consisting of elements $g$ such that $g\Gamma g^{-1}$ is \textit{commensurable} with $\Gamma$, which is to say that $g\Gamma g^{-1} \cap \Gamma$ is a finite index subgroup of $\Gamma$ and $g\Gamma g^{-1}$. Note that $\Gamma < \mbox{comm}(\Gamma)$. 

Let $\Lambda_{i}$ denote the full isometry group of the tiling $T_{i}$; note that $\Gamma<T_{i}$ for each $i$, and $\Gamma$ acts freely and transitively on the set of tiles. Thus $\Gamma$ has finite index in $T_{i}$, and therefore $\mbox{comm}(\Gamma)= \mbox{comm}(\Lambda_{i})$. 

\begin{lemma} \label{lm:12}
$r \leq [\mbox{comm}(\Lambda_{1}): \Lambda_{1}]$. 
\end{lemma}
\begin{proof} For each $k \in \left\{1,...,r\right\}$, let $g_{k} \in \mbox{Isom}(\mathbb{H}^{2})$ be an element sending $T_{1}$ to $T_{k}$. Then $g_{k} \Lambda_{1} g_{k}^{-1}$ is the group of isometries of $T_{k}$, and by assumption this is also $\Lambda_{k}$; $\Lambda_{k} \cap \Lambda_{1}$ contains $\Gamma$ which is finite index in both, and therefore $g_{k} \in \mbox{comm}(\Lambda_{1})$. 

Suppose for $j\neq k$ that $g_{k}, g_{j}$ are in the same coset of $\Lambda_{1}$ in $\mbox{comm}(\Lambda_{1})$. Then $g_{j}^{-1}g_{k}$ fixes $T_{1}$. Thus $g_{j}^{-1}$ sends $T_{k}$ to $T_{1}$, but by definition it also sends $T_{j}$ to $T_{1}$, a contradiction unless $T_{k}= T_{j}$ or $g_{j}= g_{k}$, either of which implies the other, and that $j=k$. 
\end{proof}

\begin{lemma} \label{lm:13}
For $g$ sufficiently large, $\mbox{comm}(\Lambda_{1})$ is a discrete subgroup of $\mbox{Isom}(\mathbb{H}^{2})$. 
\end{lemma}

\begin{proof} A fundamental polygon for the tiling $T_{1}$ can be subdivided into $16g-8$ copies of a hyperbolic triangle with angles $\frac{2\pi}{16g-8}, \frac{\pi}{4}, \frac{\pi}{2}$. As a consequence, $\Lambda_{1}$ is commensurable with the triangle group of signature $(16g-8,8,4)$, the subgroup of $\mbox{Isom}(\mathbb{H}^{2})$ generated by reflections in the sides of one of these triangles, and for which a presentation is as follows:
\[ \langle a,b,c | a^{2}=b^{2}=c^{2}= (ab)^{8}=(bc)^{4}=(ca)^{16g-8}=1 \rangle.\]
Denote this group by $T_{g}$. By a result of Long-Machlachlan-Reid \cite{LMR} (also Takeuchi \cite{Tak}), there are only finitely many arithmetic triangle groups. By a theorem of Margulis \cite{Marg}, a subgroup of $\mbox{Isom}(\mathbb{H}^{2})$ is arithmetic if and only if its commensurator is not discrete. Therefore, for all sufficiently large $g$, $\mbox{comm}(T_{g})$ is discrete. 

Since $\Lambda_{1}$ is commensurable with $T_{g}$, $\mbox{comm}(\Lambda_{1})=\mbox{comm}(T_{g})$, and is therefore also discrete. 

\end{proof}

Using Lemma \ref{lm:13},  for sufficiently large $g$, we can assume that the quotient space $\mathbb{H}^{2}/(\mbox{comm}(\Lambda_{1})$ is a hyperbolic orbifold $\mathcal{O}_{g}$, and therefore the area of $\mathcal{O}_{g}$ is at least $\pi/21$ (\cite{Far-Mar}). $\mathcal{O}_{g}$ is covered by the orbifold $\mathbb{H}^{2}/\Lambda_{1}$, and the degree of the covering map is equal to the index $[\mbox{comm}(\Lambda_{1}):\Lambda_{1}]$, and is also equal to the ratio of the area of $\mathbb{H}^{2}/\Lambda_{1}$ to the area of $\mathcal{O}_{g}$. The area of a fundamental domain for the action of $\Lambda_{1}$ is bounded above by the area of $\sigma$, and therefore 
\[ r\leq  \frac{2\pi(2g-2)}{\pi/21 }= 42(2g-2). \]

Thus, at most $42(2g-2)$ minimal length, minimally intersecting filling pairs can coincide on the same hyperbolic surface, and hence there are at least $N(g)/(42(2g-2))$ surfaces containing such a pair. Since $N(g)$ grows exponentially by Theorem \ref{thm:1}, so must the number of such surfaces. 

\end{proof}

\begin{remark}\label{rk:4}
Note that, more generally, we can also consider the functions $\mathcal{F}^{(k)}_{g}:\mathcal{M}_{g}\rightarrow \mathbb{R}$, which, given a hyperbolic metric $\sigma$ on $S_{g}$, outputs the length of the shortest filling pair whose complement has $k$ or fewer complementary regions. In this notation, $\mathcal{F}_{g}= \mathcal{F}^{(1)}_{g}$.

\begin{figure}[H]
\centering
	\includegraphics[width=3.5in]{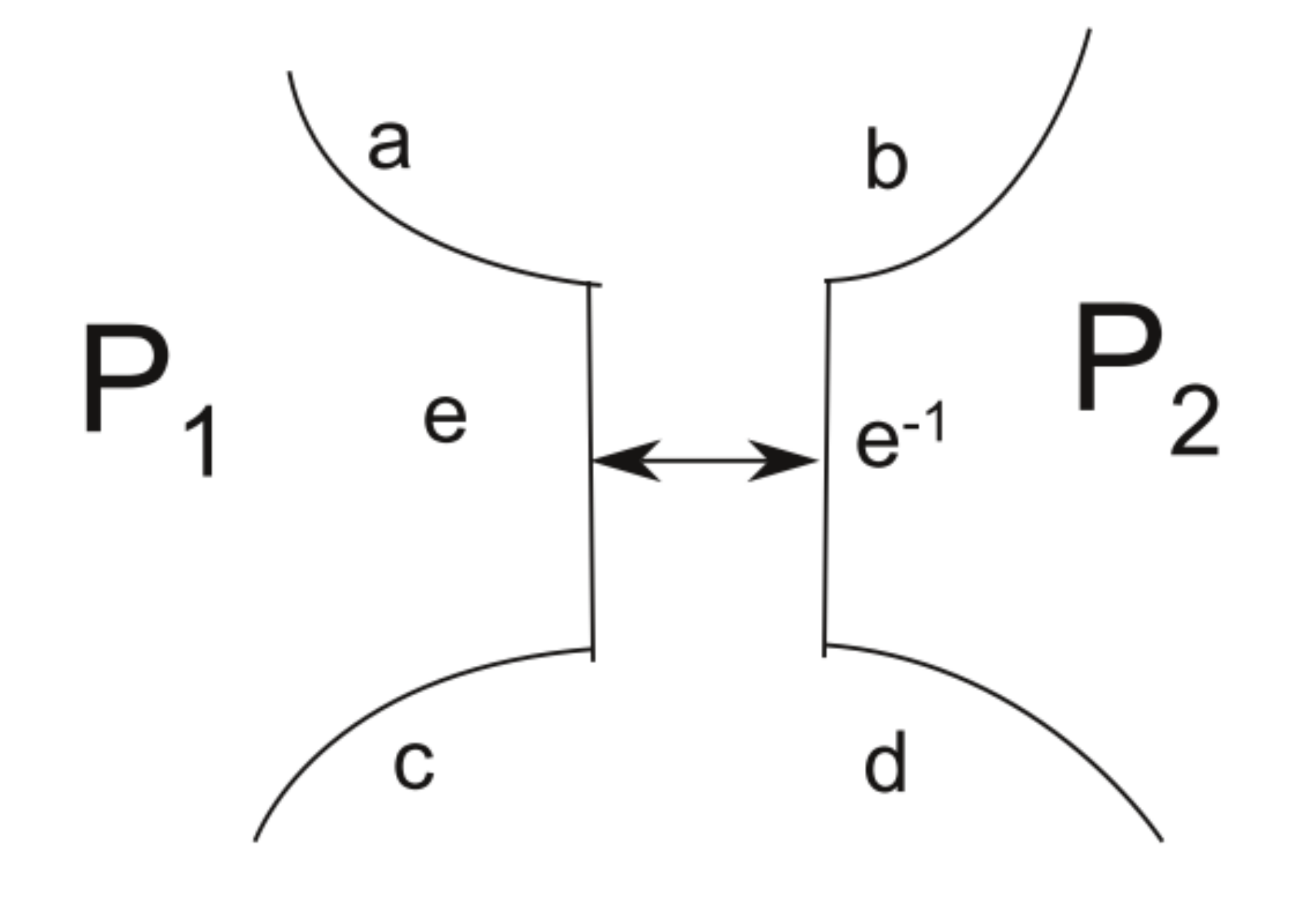}
\caption{A schematic picture of the formation of $P_{1,2}$ from $P_{1}$ and $P_{2}$. We lose a total of $4$ edges in the creation of $P_{1,2}$: $e,e^{-1}$, and $a$ and $b$ concatenate smoothly because they represent adjacent arcs on the same geodesic, and therefore they are replaced by a single edge. By the same argument, so are $c$ and $d$. }
\end{figure}

We claim that $\mathcal{F}^{(2)}_{g}$ has the same minima as $\mathcal{F}^{(1)}_{g}$. To see this, let $(\alpha,\beta)$ be a filling pair of geodesics on $\sigma$ intersecting $2g$ times, so that cutting along $\alpha \cup \beta$ produces a pair $P_{1},P_{2}$ of hyperbolic polygons. Between both polygons, there is a total of $8g$ edges.

Then choose an edge $e$ on $P_{1}$ which glues to an edge $e^{-1}$ on $P_{2}$ as in Figure $9$. Since $e$ and $e^{-1}$ have the same hyperbolic length, we can glue $P_{1}$ and $P_{2}$ together along $e$ to produce a new polygon $P_{1,2}$, which must have $4$ fewer sides than the total number of sides in $P_{1}\sqcup P_{2}$, for we lose the two edges $e,e^{-1}$ which lie in the interior of $P_{1,2}$, and the edges $a, b$ concatenate together smoothly in $P_{1,2}$ to form a single long edge since they correspond to adjacent arcs along the same geodesic on $\sigma$. The same is true for the edges $c,d$. 

Thus the total perimeter of $P_{1} \sqcup P_{2}$ is strictly larger than the perimeter of a regular, right-angled $(8g-4)$- gon, and therefore we've shown:

\begin{corollary} For all $\sigma \in \mathcal{M}_{g}$, $\mathcal{F}^{(2)}_{g} \geq m_{g}/2$, and the same hyperbolic surfaces minimize both $\mathcal{F}^{(1)}_{g}$ and $\mathcal{F}^{(2)}_{g}$. Therefore, the entirety of Theorem $1.4$ also holds for $\mathcal{F}^{(2)}_{g}$. 
\end{corollary}
\end{remark}

Motivated by Corollary $4.4$, we make the following conjecture:

\begin{conjecture} \label{con:1}
Define $\mathcal{Y}_{g}:\mathcal{M}_{g}\rightarrow \mathbb{R}$ to be the ``filling pair systole'' function, which outputs the length of the shortest filling pair. Then $\mathcal{Y}_{g}\geq m_{g}/2$.

\end{conjecture}


\begin{thebibliography}{99}

\bibitem{Akro}
H. ~Akrout. \emph{singularit\'{e}s topologiques des systoles g\'{e}n\'{e}ralis\'{e}es}. Topology 42(2), 291-308, 2003. 

\bibitem{And-Par-Pet}
J. ~Anderson, H. ~Parlier, A. ~Pettet. \emph{Small filling sets of curves on a surface}. Topology Appl. 158, 1, 84-92, 2011. 

\bibitem{Bez}
K. ~Bezdek. \emph{Ein elementarer Beweis f\"{u}r die isoperimetrische Ungleichung in der euklidischen und hyperbolischen Ebene.} Ann. Univ. Sci. Budap. Rolando E\"{o}tv\"{o}s, Sect. Math 27 (1984), 107-112.

\bibitem{Bus}
P. ~Buser. \emph{Geometry and Spectra of Compact Riemann Surfaces}. Progress in Mathematics, vol. 106, Birkh\"{a}user Boston Inc., Bostom, MA, 1992. 

\bibitem{Bus-Sar} 
P. ~Buser, P. ~Sarnak. \emph{On the period matrix of a Riemann surface of large genus}. Invent. Math., 117(1):27-56, 1994. 


\bibitem{Far-Mar}
B.~Farb, D. ~Margalit. \emph{A Primer on Mapping Class Groups}, volume $49$ of $\textit{Princeton Mathematical Series}$. Princeton University Press, Princeton, NJ, 2012. ISBN 9780691147949.

\bibitem{LMR}
D. D. ~Long, C. ~Maclachlan, A. ~Reid, \emph{Arithmetic fuchsian groups of genus zero}, Pure Appl. Math. Q. \textbf{2} (2006), 2, 569-599.

\bibitem{Marg}
G. ~Margulis. \emph{Discrete subgroups of semi-simple Lie groups}, Ergeb. Der. Math. \textbf{17}, Springer-Verlag (1989).

\bibitem{Mum}
D. ~Mumford. \emph{A remark on Mahler's compactness theorem}, Proceedings of the American Mathematical Society \textbf{28} (1971), 289-294.

\bibitem{Par1} 
H. ~Parlier. \emph{The homology systole of hyperbolic Riemann surfaces}. Geom. Dedicata 157(1), 331-338, 2012. 

\bibitem{Par2} 
H. ~Parlier. \emph{Bers' constants for punctured spheres and hyperelliptic surfaces}. J. Topol. Anal., 4(3), 271-296, 2012. 

\bibitem{Schal1}
P. ~Schmutz Schaller. \emph{Riemann surfaces with shortest geodesic of maximal length}. Geom. Funct. Anal., 3(6):564-631, 1993. 

\bibitem{Schal2}
P. ~Schmutz Schaller. \emph{Geometry of Riemann surfaces based on closed geodesics}. Bull. Ame. Math. Soc. 35(3): 193-214, 1998. 

\bibitem{Schal3} 
P. ~Schmutz Schaller. \emph{Systoles and Topological Morse Functions for Riemann surfaces}. J. Differ. Geom. 52(3), 407-452, 1999. 

\bibitem{Tak}
K. ~Takeuchi. \emph{Commensurability classes of arithmetic triangle groups}, J. Fac. Sci. Univ. Tokyo \textbf{24} (1977), 201-222. 



\end{thebibliography}
\end{document}